\documentclass[11pt]{article}
\newbox\mybox 
\newdimen\myboxwidth    

\RequirePackage[preprint, aop,noinfoline]{imsart}
\usepackage{hyperref}
\usepackage[letterpaper,margin=2.5cm]{geometry}
 \usepackage{stmaryrd, mathrsfs}
 \usepackage{enumerate,mathtools}
\usepackage{amsmath,amsfonts,amsthm,amssymb, graphicx, dsfont,wrapfig, caption, subcaption}
\usepackage{constants}
\usepackage[usenames,dvipsnames]{xcolor}	

\usepackage[initials]{amsrefs}

\newcommand{\va}{\mathbf{a}}\newcommand{\vb}{\mathbf{b}}
\newcommand{\ve}{\mathbf{e}}

\newcommand{\vv}{\mathbf{v}}\newcommand{\vw}{\mathbf{w}}\newcommand{\vx}{\mathbf{x}}
\newcommand{\vy}{\mathbf{y}}\newcommand{\vz}{\mathbf{z}}

\newcommand{\dZ}{\mathds{Z}}

\newcommand{\norm}[1]{\left\Vert#1\right\Vert}

\newenvironment{enumerateA}{\begin{enumerate}[\upshape (A)]\setlength{\itemsep}{4pt}}{\end{enumerate}}
\DeclareMathOperator{\pr}{\mathds{P}}
\newcommand{\eg}{\emph{e.g.,}\ }

\newcommand{\Z}{\mathds{Z}}
\newcommand{\Zd}{\mathds{Z}^d}
\newcommand{\Rd}{\mathds{R}^d}
\newcommand{\bbR}{\mathds{R}}
\newcommand{\tz}{\tilde{\vz}}
\newcommand{\wtD}{\widetilde{D}}

\newcommand{\E}{\mathds{E}}
\newcommand{\prob}{\mathds{P}}

\newcommand{\cC}{\mathcal{C}}
\newcommand{\cD}{\mathcal{D}}
\newcommand{\cT}{\mathcal{T}}
\newcommand{\cE}{\mathcal{E}}
\newcommand{\indi}{\mathbf{1}}

\newcommand{\send}{\mathsf{send}}
\newcommand{\get}{\mathsf{get}}

\newcommand{\mass}{\mathfrak{m}}
\newcommand{\Anns}{Ann'}

\newtheorem{thm}{Theorem}
\newtheorem{lem}[thm]{Lemma}
\newtheorem{prop}[thm]{Proposition}
\newtheorem{defin}{Definition}

\newtheorem{clam}[thm]{Claim}

\newcommand{\chJH}[1]{{\color{black}#1}}

\def\beq{ \begin{equation} }
 \def\eeq{ \end{equation} }
 \def\beqx{ \begin{equation*} }
 \def\eeqx{ \end{equation*} }
 \def\beqa{\begin{eqnarray}}
 \def\eeqa{\end{eqnarray}}
 \def\beqax{\begin{eqnarray*}}
 \def\eeqax{\end{eqnarray*}}

\newcommand{\sa}[1]{\ensuremath{\,{\buildrel #1 \over \longleftrightarrow}\,}}
\newcommand{\sS}{\mathscr{S}}
\newcommand{\dN}{\mathds{N}}
\let\del=\partial

\let\gc=\gamma
\let\lra=\leftrightarrow
\let\cons=c
\let\Cons=C
\newconstantfamily{sma}{symbol=a}
\newconstantfamily{lga}{symbol=A}

\newcommand{\rrb}{\rrbracket}
\newcommand{\llb}{\llbracket}
\newcommand{\br}[1]{\llb#1\rrb}

\setattribute{journal}{name}{}

\begin{document}
\begin{frontmatter}
	\title{Restricted percolation critical exponents in high dimensions}
	\runtitle{Restricted High Dimensional Percolation Exponents}
	
	\begin{aug}
		\author{\fnms{Shirshendu} \snm{Chatterjee}\thanksref{m1,m2,t1}\ead[label=e1]{shirshendu@ccny.cuny.edu}  \ead[label=u1,url]{http://shirshendu.ccny.cuny.edu/}}
		\and
		\author{\fnms{Jack} \snm{Hanson}\thanksref{m2,t2}\ead[label=e2]{jhanson@ccny.cuny.edu}
			\ead[label=u2,url]{http://jhanson.ccny.cuny.edu/}}
		
		\thankstext{t1}{This work was supported by a grant from the Simons Foundation (\#430073, Shirshendu Chatterjee).}
		\thankstext{t2}{Funded in part by NSF Grant DMS-1612921.}

		\runauthor{Chatterjee and Hanson}
		
		\affiliation{City University of New York; City College  \thanksmark{m2}  and 
			 Graduate Center\thanksmark{m1}}
		
		\address{Shirshendu Chatterjee\\Department of Mathematics \\City University of New York, City College \\ 160 Convent Ave, NAC 4/114B \\  New York, NY 10031 , USA\\
			\printead{e1}\\
			\printead{u1}}
		\address{Jack Hanson \\ Department of Mathematics \\City University of New York, City College \\ 160 Convent Ave,  NAC 6/292 \\ New York, NY 10031 , USA\\
			\printead{e2}\\
			\printead{u2}}
	\end{aug}
	
	\date{\today}
	\begin{abstract}
		Despite great progress in the study of critical percolation on $\Zd$ for $d$ large, properties of critical clusters in high-dimensional fractional spaces and boxes remain poorly understood, unlike the situation in two dimensions. Closely related models such as critical branching random walk give natural conjectures for the value of the relevant high-dimensional critical exponents; see in particular the conjecture by Kozma-Nachmias that the probability that $0$ and $(n, n, n, \ldots)$ are connected within $[-n,n]^d$ scales as $n^{-2-2d}$. 

In this paper, we study the properties of critical clusters in high-dimensional half-spaces and boxes. In half-spaces, we show that the probability of an open connection (``arm'') from $0$ to the boundary of a sidelength $n$ box scales as $n^{-3}$. We also find the scaling of the half-space two-point function (the probability of an open connection between two vertices) and the tail of the cluster size distribution. In boxes, we obtain the scaling of the two-point function between vertices which are any macroscopic distance away from the boundary.
	\end{abstract}
	
	\begin{keyword}[class=AMS]
		\kwd[Primary ]{60K35}\kwd[; secondary ]{82B43}
	\end{keyword}
	
	\begin{keyword}
		\kwd{percolation} \kwd{critical percolation} \kwd{percolation in high dimension}  \kwd{critical exponents}  \kwd{one arm exponent} \kwd{two point function}
	\end{keyword}
\end{frontmatter}

\section{Introduction}\label{sec:int}
In this paper, we consider the {\it bond percolation} model on the canonical $d$-dimensional {\it half space}, which is the subgraph of the $d$-dimensional lattice induced by the vertices in $\{\vx\in\dZ^d: x(1)\ge 0\}$. It is well known \cite{BGN91} that there is no infinite {\it open cluster} almost surely in critical percolation on any $d$-dimensional half spaces for any $d>1$, although the analogous problem for $d$-dimensional lattices is settled only for $d=2$ (due to Harris \cite{H60} and Kesten \cite{K80}) and in {\it high dimensions} (due to Hara \& Slade \cite{HS90} and Fitzner \& van der Hofstad \cite{FH17}), but is still open  for the intermediate dimensions. By {\it high dimensions} we refer to one of the two underlying graphs: (i) the {\it square lattice} $\dZ^d$ (i.e.~the graph with vertex set $\dZ^d$ such that $\vx, \vy\in\dZ^d$ are neighbors iff $\norm{\vx-\vy}_1=1$) with $d\ge 11$ or, (ii) the {\it spread out lattice} (i.e.~the graph with vertex set $\dZ^d$ such that $\vx, \vy\in\dZ^d$ are neighbors iff $\norm{\vx-\vy}_\infty\le L$ for sufficiently large $L$) with $d > 6$ (see further definitions below).	

The results of \cite{BGN91} mentioned above lead to questions about the main features of critical percolation clusters within half spaces, including the behavior of \\
(a)  the {\it one arm probability}, which is the probability that the origin is connected
to the boundary of the ball having $\ell^\infty$ radius $n$ lying within the half-space;\\
(b)  the {\it two point function} $\tau_H(\vx,\vy)$, which is the probability that two vertices $\vx$ and $\vy$ are  connected by an open path lying within the half-space;\\
(c) the {\it upper tail of the cluster size}, which is the probability that the cardinality of the half-space open cluster containing the origin is larger than $n$.\\
Clearly, the probabilities in (a) and (c) (resp.~(b)) tend to 0 as $n$ (resp.~$\norm{\vx-\vy}_\infty$) tends to $\infty$.

It is widely believed among the physicists (see e.g.~\cite[Section 2.2]{HH17})  that the analogous probabilities for critical percolation on lattices --- and fractional spaces including half-spaces ---  decay polynomially, i.e.~the probabilities in (a), (b) and (c) are $n^{-1/\rho+o(1)},\, \norm{\vx-\vy}_\infty^{2-d+\eta+o(1)}$ and $n^{-1/\zeta+o(1)}$  respectively for some critical exponents $\rho, \eta, \zeta$, which  depend only  on the dimension of the underlying lattice rather than the structural details of it.

In this paper, we have obtained the above mentioned critical exponents for  high dimensional half spaces.

\begin{thm} \label{Critical Exponents}
	Consider critical percolation on high dimensional half spaces. Then
	\begin{align*}
	(a) & \text{ one arm probability }  \asymp n^{-3}, \\
	(b) & \text{ two point function  } \tau_H(\vx,\vy)  \asymp 
	\begin{cases} \|\vx - \vy\|_{\infty}^{2-d} & \text{ if both $x(1)$ and $y(1)$ are $O(\norm{\vx-\vy}_\infty)$} \\
	\|\vx - \vy\|_{\infty}^{1-d} & \text{ if $x(1)=0$ and $y(1)$ is $O(\norm{\vx-\vy}_\infty)$} \\
	\|\vx - \vy\|_{\infty}^{-d} & \text{ if both $x(1)$ and $y(1)$ are 0}, 
	\end{cases} \text{ and }
	\\
	(c) & \text{ upper tail of cluster size }  \asymp n^{-3/4}.
	\end{align*}
\end{thm}
Here and later, we write $f(n) \asymp g(n)$ to mean that there is a constant $C > 0$ (possibly depending on the dimension $d$ and the choice of the lattice) such that $C^{-1}f(n) \le g(n) \le Cf(n)$ for all $n \geq 1$.  

In the past, the critical exponents $\rho,\eta, \zeta$ (as described above)  were obtained for some other graphs. It is known that
\begin{enumerateA}
	\item $\rho=1$ for critical percolation  on regular infinite trees (due to Kolmogorov \cite{K38}), $\rho=1$
	for critical oriented percolation on spread-out lattices having dimension larger than $4$ (due to van der Hofstad, den Hollander \& Slade \cite{HHS07, HHS07a}),
	$\rho=1/2$ for high dimensional lattices (due to Sakai \cite{S04}, Kozma \& Nachmias \cite{KN11}),
	$\rho=48/5$ for  critical site percolation on the two-dimensional triangular lattice 
	(due to Lawler, Schramm \& Werner \cite{LSW02}).
	\item $\eta=0$  for  critical percolation on high dimensional lattices (due to Hara \cite{H08} and Hara \& Slade \cite{HS90}) and $\eta = -5/24$ on the two-dimensional triangular lattice (due to Lawler, Schramm \& Werner \cite{LSW02} and to Kesten \cite[Equation (4)]{K87}).
	\item $\zeta$ equals 2  for critical percolation on high dimensional lattices (due to  Aizenman \& Barsky \cite{AB87}, Barsky \& Aizenman \cite{BA91} and Hara \& Slade \cite{HS90}), and $\zeta = 91/5$ on the two-dimensional triangular lattice (due to Lawler, Schramm, \& Werner \cite{LSW02} and to Kesten \cite[Equation (9)]{K87}).
\end{enumerateA}
It would be interesting to analyze critical percolation clusters on other high-dimensional fractional lattices.
In the following section we describe the background and motivation for our paper in more detail.

\subsection{Background and motivation}
Over the last few decades, there has been a great deal of research into the existence and properties of phase transitions in different statistical-mechanical models. The simplest among such models is perhaps the Bernoulli bond percolation model, where one obtains a random graph from an underlying infinite base graph $G$ by independently retaining each of its edges with probability $p\in[0,1]$ and deleting it with probability  $1-p$.  For an introduction to the subject and for earlier works, when the base graph  is $\dZ^d$ with nearest-neighbor edges, we recommend \cite{G99}. See also \cite[Chapter 7]{LP17} for the treatment of percolation on general transitive graphs including homogeneous trees. 

We write $\pr_p$  for the probability measure on subgraphs of $G$ obtained as above.
Edges retained are called {\it open} and edges deleted are called {\it closed}. The
critical percolation probability $p_c$ is defined by
\begin{equation} p_c := \inf\left\{p: \pr_p(\text{at least one of the components of the open subgraph is infinite})>0\right\}. \label{eq:pcdef}\end{equation}
If $G$ is a lattice, then for $p<p_c$, which is called the {\it subcritical regime},  there is no infinite cluster almost surely, and for $p>p_c$, which is called the {\it supercritical regime}, there is one  infinite cluster. Properties of both the subcritical and the supercritical clusters are well understood \cite{G99}.  We also have a fairly good understanding of critical percolation in two dimensions and high dimensions.  On the contrary, not much is known for the intermediate dimensions. It is not even clear whether there is an infinite component at $p_c$. 

In this paper, we consider critical percolation in high dimensional half spaces in high dimensions.  One of the direct motivations for considering critical percolation on fractional lattices is the conjecture \cite[page 378]{KN11} that the 
\[ \text{{\it corner one arm } probability } \pr_{p_c}\left(0\sa{B(n)} (n, n, \ldots, n)\right) \asymp n^{\xi(d)}, \text{ where } \xi(d)=2-2d\]
in high dimensions. Here and later we write $\vx\sa{S} \vy$ to denote the event that $\vx$ is connected to $\vy$ by an open path staying within $S$, and $B(n)$ denotes the box $[-n,n]^d$. In order to prove this conjecture one needs a clear understanding of critical percolation in the fractional spaces. One of the main difficulties here is the reduction of symmetry in case of fractional spaces. The  techniques based on ``lace expansion", which are used to determine the behavior of the two point function for high dimensional lattices, use translation invariance and hence the symmetry of the lattices heavily.

It is also widely believed (see e.g.~\cite[Section 2.2]{HH17})
that for high dimensional lattices the behavior of {\it critical Branching Random Walk} (BRW) is closely related to that of critical percolation.   More formally, the critical exponents, which describe the ``shape" of the clusters, for the two models are expected to attain the same values. 
In particular,  the values of the critical exponents $(\rho, \eta, \zeta, \xi)$ for critical percolation on high dimensional lattices (resp.~fractional spaces)
should match with the values of $(\rho, \eta, \zeta, \xi)$  for critical BRW on high dimensional lattices (resp.~the BRW killed at the boundary of the corresponding fractional space).
The values of  $\rho, \eta$ and $\zeta$ are known for both critical percolation and critical BRW on high dimensional lattices, and the values agree. The critical exponents for the two-point function associated with critical BRW on fractional-spaces are readily computable; in particular, the half-space critical exponent
is $1-d$. It is natural to try to find the critical exponents for critical percolation on fractional-spaces and to compare with the values for critical BRW.

While we are unaware of past work on the critical exponents for critical percolation in
high-dimensional half-spaces, analogous problems have been studied for decades in two
dimensions. Early on, Kesten \& Zhang \cite{KZ87} considered critical percolation on
the two dimensional (angular) fractional space $A_\varphi:=\{(r\cos\theta, r\sin\theta)\in\dZ^2: r\ge 0, \theta\le \varphi\}, \varphi \in [0, 2\pi],$ and showed that the one arm exponent $\rho(\varphi)$ is strictly monotone in $\varphi$.
Using classical methods, one can compute some half-plane critical exponents, such as the ``polychromatic three-arm half-plane exponent" 
(see the lecture notes on two dimensional critical percolation \cite{W09}). These methods are ad hoc, but universal. 
Much later, SLE based methods were developed to study critical percolation on the two-dimensional triangular lattice during the last two decades. These methods have enabled researchers to compute most critical exponents  for critical percolation on half-planes and two-dimensional fractional-spaces of the triangular lattice \cite[Section 3]{SW01}.   
In the case of two dimensional lattices, the two point function critical exponent (and similarly the cluster size exponent) can be derived from the one arm probability using techniques of ``gluing'' \cite{K87}. These gluing arguments also give the asymptotic behavior of the {\it restricted two-point function} $\tau_B(\vx,\vy):=\pr_{p_c}(\vx\sa{B}\vy)$, where $B$ is a two dimensional box and $\vx, \vy\in B$.  In two dimensions, the
restricted two-point function $\tau_B(\vx,\vy)$  within a box $B$ scales like the unrestricted two-point function  $\tau(\vx,\vy)$  as long as the two
points $\vx$ and $\vy$ are far from the boundary $\del B$ of the box $B$.

Unlike the two dimensional lattices, a major difficulty in analyzing critical high dimensional percolation is that we can only control the open clusters in a very indirect way. In case of high dimensions, the gluing  techniques (as mentioned above) do not work even in principle because of the diverging number of spanning clusters, and an analogue of SLE is currently unavailable. As a result,  it took so long (twenty years) to get the one arm exponent  in high dimensional critical percolation from the corresponding cluster size exponent. 

In 1990, a breakthrough was made to bound the two-point function \cite{HS90} using {\it lace expansion}. 
The bound, the so-called \emph{infrared bound}, established the so-called ``triangle condition" and thereby completed the argument of Barsky \& Aizenman \cite{BA91} to obtain the cluster size exponent. The infrared bound was later strengthened to give $\eta = 0$ (see \cite{H08}).
On the other hand, the one arm exponent was obtained (i) first under some unproven assumptions in 2004 \cite{S04}, (ii) without any unproven assumptions in 2011 \cite{KN11}.   

Another of our main results says that, in high dimensions, $\tau_B(\vx,\vy)$ scales as $\tau(\vx,\vy)$, for $\vx$ and $\vy$ far from the boundary $\del B$ of the set $B$ (analogously to the two-dimensional result mentioned above). This result is crucial for our proof of Theorem \ref{Critical Exponents}; we also believe it is interesting in its own right and is a potential tool for studying other properties of open clusters. Our proof is very different from the proof of the two-dimensional analogue, since gluing methods are unavailable in high dimensions.
\begin{thm}\label{thm:boxcon}
	Suppose $M>1$ is any constant. In high dimensions, there are constants $C>c>0$ (depending on $M$ and $d$ only) such that for all $n$ and for all $\vx \neq \vy\in B(n)$, 
	\[ c\norm{\vx-\vy}_{\infty}^{2-d} \le \pr_{p_c}\left(\vx\sa{B(Mn)}\vy\right) \le C\norm{\vx-\vy}_{\infty}^{2-d}.\]
\end{thm}
Our results hold in high dimensions; the condition that $d \geq 11$ could be relaxed to $d > 6$ if one could prove that the cluster size and two-point  function satisfy
\beqa
\pr_{p_c}(\#\{\vx\in \dZ^d: 0\lra \vx\}>n) \asymp n^{-1/2} \label{SizeEst} \\
\pr_{p_c}(\vx\lra \vy)\asymp \norm{\vx-\vy}^{2-d}. \label{2PtEst}
\eeqa
For high dimensional lattices, \eqref{SizeEst} was established  in \cite{AB87, BA91, HS90}, and \eqref{2PtEst} was proved in \cite{H08, HS90, FH17} (for nearest-neighbor  lattices) and \cite{HHS03} (for spread-out lattices). Like many models of statistical physics,  the critical exponents for critical percolation on $\dZ^d$ are expected to  attain the same value as they do on an infinite regular tree for all $d$ large enough. 
For example,  it is well known that \eqref{SizeEst} holds for critical percolation on an infinite regular tree \cite{AN27}, and the authors of  \cite{CC87} have worked on other aspects  of the ``tree-like'' behavior of the high dimensional lattices. 
The dimension at which the ``tree-like" behavior starts to occur  is often called the {\it upper critical dimension}. It is believed that the upper critical dimension for critical percolation  is 6, so \eqref{SizeEst} is expected to hold for all lattices with dimension larger than 6. So far, it is only known to hold in high dimensions. 

Other than the critical exponents discussed above, researchers in percolation theory have also worked on other aspects of critical percolation clusters, including spanning clusters within cubes \cite{A97}, scaling limits for critical percolation on $\dZ^d$ (see \eg  \cite{HS00, HS00a, S01, AB99}), 
size of the intrinsic balls \cite{KN09, S10}, and structural properties of high dimensional percolation clusters on tori  \cite{HS14}. We mention also the non-backtracking lace expansion \cite{FH17}, which aims to prove that the critical exponents attain the same value for all dimensions higher than the upper critical dimension.

\subsection{Outline of the proof}
The first result proved is a lower bound on the half-space one-arm probability --- which establishes a portion of (a) from Theorem \ref{Critical Exponents} --- in Section \ref{sec:pihlb}. We argue by showing that with uniformly positive probability, there are at least order $n^{d-4}$ vertices on the boundary of a sidelength $n$ box having arms across this box (which are necessarily half-space arms). The proof is via a second-moment argument on a suitably defined set of sufficiently regular spanning clusters $\sS$. Since the expected number of vertices on the boundary having such arms is at most $n^{d-1}$ times the half-space one-arm probability, the bound follows. This argument does not depend on Theorem \ref{thm:boxcon} or the remainder of Theorem \ref{Critical Exponents}.

The remaining arguments rely on Theorem \ref{thm:boxcon}, and so we prove it next (in Section \ref{sec:rtwopt}). It is based on an iterative improvement of the following form: assume that for some $M > 1$, we have $\tau_{B(Mn)}(0, \vx) \geq c \|\vx\|^{2-d}$ uniformly in $n$ and in $\vx \in B(n)$. Then $\tau_{B((M+1)n/2)}(0, \vx) \geq c' \|\vx\|^{2-d}$ uniformly in $n$ and $\vx \in B(n)$, for some $c' > 0$. To show this iterative improvement, we use the inductive hypothesis to build connections from $\vx$ to $0$ lying in $B((M+1)n/2)$. The key is conditioning on $\vx$ having an arm to distance $(M-1)n/2$ directed away from the boundary of the large box; the endpoint of this arm is farther from $\partial B((M+1)n/2)$ than $\vx$, and it can thus be extended to $0$ using the bound on $\tau_{B(Mn)}$.
 
We next upper-bound the one-arm probability in (a) from Theorem \ref{Critical Exponents}. Letting $\pi_H(n)$ be the half-space one-arm probability to distance $n$, we bound $\pi_H(2n)$ in terms of $\pi_H(n)$. Conditional on an arm to distance $n$, we show that either the arm is likely to go extinct before reaching distance $2n$ (corresponding to a small contribution to $\pi_H(2n)$) or $0$ is typically connected, by open paths avoiding the box $B_{-}(n/2) = [-1-n/2,-1] \times[-n/2,n/2]^{d-1}$, to order $n^4$ vertices having $\ell^\infty$ norm of order $n$.  The latter probability is shown to be small by a mass-transport argument. Roughly, if $0$ were typically connected to too many vertices at distance $n$ by paths avoiding $B_{-}(n)$, then $n \ve_1$ would typically be connected to many vertices on the boundary of $B_{-}(n)$. This would mean that the $\Zd$ open cluster of $n \ve_1$ is typically very large, in contradiction to existing bounds.

Parts (b) and (c) of Theorem \ref{Critical Exponents} use part (a) as input; this is in contrast to $\Zd$, where the values of $\eta$ and $\zeta$ were found first and used to show $\rho = 1/2$. To show the bounds for $\tau_H$, we show that conditional on $0$ having a half-space arm to distance $n$, typically $0$ is connected to order $n^2$ vertices on the top of the sidelength $n$ half-space box. The probability of further connection is now estimated using techniques like those used to prove Theorem \ref{thm:boxcon}. The cluster size is now controlled using the arm probability from (a) and moment bounds using the estimates on $\tau_H$ from (b).

\subsection{Organization of the paper}
In Section \ref{sec:defsec}, we standardize our notation for subgraphs of $\Zd$ and basic notation and background on percolation. We then (in Section \ref{sec:masstrans}) define the mass-transport method and prove an abstract mass-transport result, Lemma \ref{lem:transport}.

In Section \ref{sec:pihlb}, we show the lower bound on the one-arm probability from (a) of Theorem \ref{Critical Exponents}: $\pi_H(n) \geq c n^{-3}.$ Section \ref{sec:extend} is devoted to results on cluster ``extensibility'' which will be crucial for proving Theorem \ref{thm:boxcon} and the upper bound on the one-arm probability. In Section \ref{sec:rtwopt}, we prove Theorem~\ref{thm:boxcon}, and in Section \ref{sec:proveext} we prove the main extensibility result from Section \ref{sec:extend}.

In Section \ref{sec:ubonearm}, we use the preceding to show $\pi_H(n) \leq C n^{-3}$, completing the proof of (a) from Theorem \ref{Critical Exponents}. This section breaks up into two parts: the choice and analysis of a particular mass-transport rule, and an iterative bound on $\pi_H$ relying on our mass-transport results. Finally, in Section \ref{sec:twoptclustend}, we bound $\tau_H$ and the tail of the cluster size distribution, proving (b) and (c) of Theorem~\ref{Critical Exponents}.

\section{Definitions and preliminary results}\label{sec:defsec}
\chJH{We will for simplicity consider explicitly the nearest-neighbor model on $\Zd$ for a fixed value of $d \geq 11$. We will not need to consider the measures $\prob_p$ for any $p$ other than $p_c$, so for the remainder of the paper we write $\prob$ for $\prob_{p_c}$. Recall that the value of $p_c$ for a $d$-dimensional half-space (defined, as on $\Zd$, via \eqref{eq:pcdef}) is the same \cite{GM90} as on $\Zd$.  }

\chJH{{\bf A note about constants:} the symbols $C$, $c$ generally represent positive constants whose values may change from line to line (and even within lines); we sometimes number them to refer to them locally. Other symbols such as $\varepsilon$ will sometimes refer to constants depending on context. When we wish to make clear the possible dependence of a constant on a parameter, we do it in a case-by-case basis, for instance by writing $C = C(K)$. Numbered constants designed to be retained on a long-term or global basis will be denoted $a_i, A_i$; certain specially labeled constants, such as $c_*$ from Theorem \ref{thm:mainextend}, will also be referred to several times throughout the paper.
}

\subsection{Graph notation}
We abuse notation and write $\Z^d$ for both the vertex set of integer vectors as well as the graph with vertex set $\Z^d$ and nearest neighbor edges. We will write $\vx \sim \vy$ if $\vx$ and $\vy$ are neighbors in $\Z^d$ -- that is, if there is an edge $e \in \Z^d$ with $e = \{\vx, \vy\}$. The norm notation $\| \vx \|$ refers to the $\ell^\infty$ norm $\|\vx \|_\infty$ unless an alternate subscript is given. \chJH{$\#A$ denotes the cardinality of a set $A$.}

  \begin{wrapfigure}{R}{0.4\textwidth}
 \begin{center}
\includegraphics[height=5cm,width=7cm]{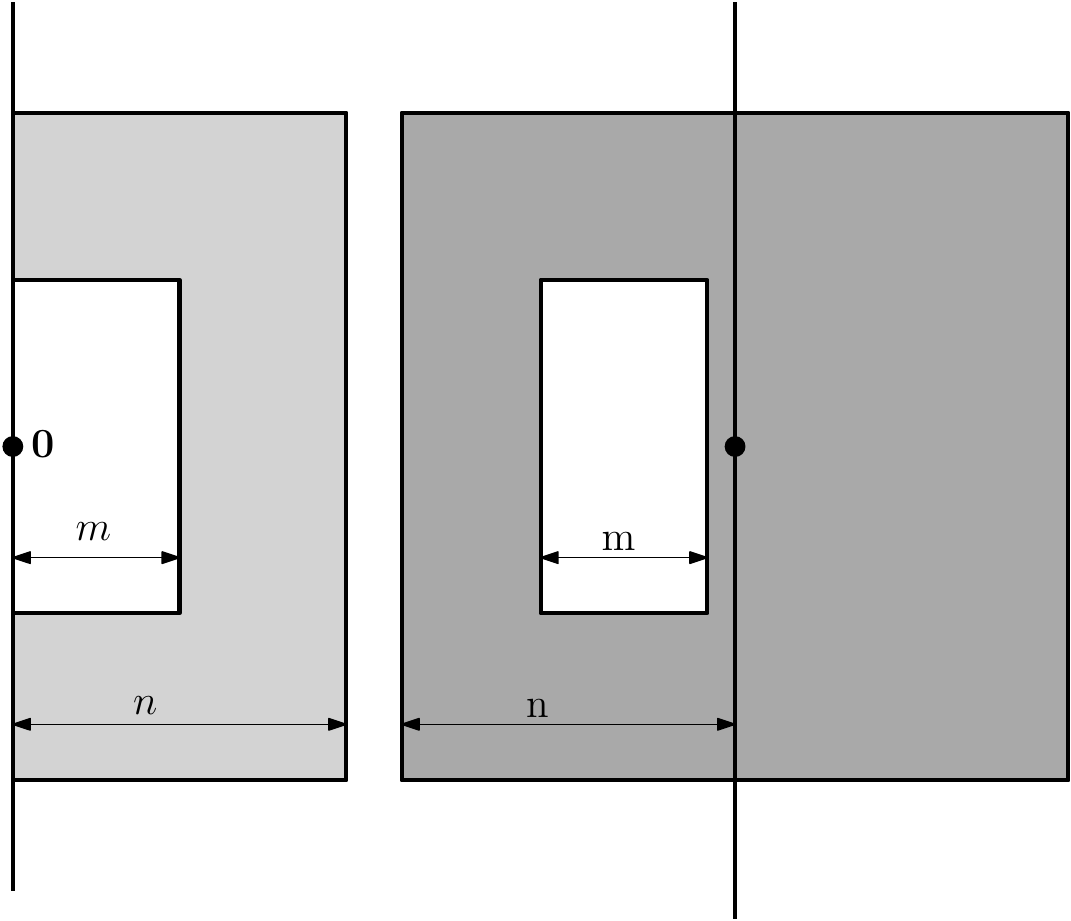} 
  \end{center}
  \caption{{\small Left:~$Ann_H(m,n)$.~Right:~$Ann'(m,n)$.}}
  \label{Fig: Ann}
\end{wrapfigure}
Let $\ve_i$ denote the $i$th standard basis vector. For $\vx \in \Zd,$ we write $x(i) = \vx \cdot \ve_i$ for $i = 1, \ldots, d$. Define the shifted half-spaces
\[\Zd_+(n) = \{\vx \in \Zd: \, x(1) \geq n\}; \]
for brevity, we write $\Z^d_+ = \Z^d_+(0)$.
The corresponding boundary hyperplane is
\[S(n) := \{\vx: \,x(1) = n \}\ . \]
The usual $\ell^\infty$-box is $B(n) := \{\vx: \, \|\vx\| \leq n\}$ with boundary $\partial B(n) = \{\vx: \, \|\vx\| = n\}$. 
Shifted boxes will be important to us; we generally denote the box centered at $\vx$ by $\vx + B(n)$. Note that the above definitions extend to non-integer values of $n$, so that for instance $B(3.5) = B(3)$. 

Generally, for a set of vertices $V$, we let $\partial V$ denote the interior vertex boundary relative to $\Zd$:
\[ \partial V = \{\vx \in V: \, \exists \vy \in \Zd \setminus V \text{ such that } \vy \sim \vx\}\ .\]
\chJH{We will sometimes need to consider boundaries relative to other subgraphs (especially $\Zd_+$). The half-space analogue of a box will be denoted $B_H(n) := B(n)\cap \Z_+^d$. The boundary of $B_H(n)$, considered as a subgraph of $\Zd_+$, is written $S'(n):= \Zd_+ \cap \partial B(n)$.}
\chJH{We also introduce $Rect(n):= [0, n] \times [-4n, 4n]^{d-1}$ as a slightly fattened version of $B_H(n)$.}

The annulus $Ann(m,n) := B(n) \setminus B(m)$. 
The corresponding half-space annuli are $Ann_H(m,n) := B_H(n) \setminus B_H(m)$.
We will often refer to shifted annuli, where one side of the inner box lies along $S(0)$. Namely, we define \chJH{$B_{-}(n) = -\ve_1 - B_H(n)$}, and (for $n \geq m$) $\Anns(m,n) = \left[B(n) \setminus B_{-}(m)\right]$ see Figure \ref{Fig: Ann}).
\chJH{The outer boundaries of annuli are defined as the vertex boundaries of their outer boxes, relative to the ambient subgraph: $\partial^+ Ann(m,n) = \partial^+\Anns(m,n) =  \partial B(n)$, and $\partial^+ Ann_H(m,n) = S'(n)$.
Similarly, the inner boundary $\partial^- Ann(m,n) = \partial B(m+1)$, with analogous definitions for the other annuli: $\partial^- \Anns(m,n) = -\ve_1 -\partial B_H(m+1)$ (where this boundary is taken relative to $\Zd$), and $\partial^- Ann_H(m,n) = S'(m+1)$.}

\chJH{We will occasionally consider graph boundaries with respect to general subgraphs of $\Zd$. If $A_0 \subseteq A_1$ are finite subsets of (the vertices of) $\Zd$, let 
	\begin{equation}
	\label{eq:genbdyv}
	\partial_{A_1} A_0 = \{\vx \in A_0: \text{ there is some } \vy \in A_1 \setminus A_0 \text{ with } \vy \sim \vx \}\ .
	\end{equation}}
\subsection{Percolation notation and previous results}
A generic percolation configuration is written $\omega = (\omega_e)_e$. The random variables $\omega_e$ are i.i.d. Bernoulli$(p_c)$, and the edge $e$ is open (resp. closed) if $\omega_e = 1$ (resp. $\omega_e = 0$). The open graph is the random subgraph of $\Zd$ whose vertex set is $\Zd$ and whose edge set is $\{e = \{\vx,\vy\}: \, \vx \sim \vy, \, \omega_e = 1\}$.  Open clusters are subgraphs of this open graph; we sometimes identify an open cluster with its vertex set.


 For $\vx, \, \vy \in \Zd$, we write $\{\vx \leftrightarrow \vy\}$ for the event that there is a path from $\vx$ to $\vy$ consisting of open edges \chJH{of $\Zd$} --- in other words, when $\vx$ and $\vy$ lie in a common open cluster \chJH{of $\Zd$}. We define $A \leftrightarrow B$ for sets $A, B \subseteq \Zd$ similarly. We denote the unique open cluster containing $\vx$ by $C(\vx):= \{\vy: \, \vx \leftrightarrow \vy\}$. 
We are interested in the probabilities of various connectivity events; for reference, we state some well-known results. By techniques of the lace expansion, it has been derived in this setting that there are constants $0<\Cl[sma]{c:twoptlow} < \Cl[lga]{c:twopthigh}<\infty$ (depending only on $d$) such that
\begin{equation}
  \label{eq:twopt}
 \Cr{c:twoptlow}\|\vx - \vy\|^{2-d}\leq  \prob\left(\vx \leftrightarrow \vy \right) \leq \Cr{c:twopthigh}\|\vx - \vy\|^{2-d} \text{ for all } \vx,\, \vy \in\Zd\ .
\end{equation}
We denote the probability (``two-point function'') appearing in \eqref{eq:twopt} by $\tau(\vx,\vy)$. \chJH{In the above and in many places where the two-point function appears, we use the convention $\|0\|^{2-d} = 1$; this minor abuse allows us to avoid some cumbersome expressions when summing products of $\tau$.}

\
\chJH{Recall the definition of the one-arm probability: the probability that a site has a connection to $\ell^\infty$ distance $n$. The result of \cite{KN11} that $\rho = 1/2$ in high dimensions mentioned above in fact shows that the one-arm probability $\pi(n)$ is asymptotic to $n^{-2}$:}
\begin{equation}
  \label{eq:onearmprob}
\Cl[sma]{c:armlow} n^{-2}\leq  \pi(n) := \prob(0 \lra \partial B(n)) \leq \Cl[lga]{c:armhigh} n^{-2} \text{ for $n \geq 1$ and  constants } 0<\Cr{c:armlow}<\Cr{c:armhigh}<\infty. 
\end{equation}

\chJH{We also need some bounds on the first and second moments of cluster sizes.}
\begin{lem}\label{lem:momentaiz}
  We have
\[\E \left[\#C(0) \cap B(n)\right] \asymp n^2; \quad \E \left[\left(\#C(0) \cap B(n)\right)^2\right] \leq C n^6\ . \]
\end{lem}
\begin{proof}
  The first moment is just $\sum_{\vx \in B(n)} \tau(0,\vx)$, and the asymptotic follows by summing \eqref{eq:twopt}. The second moment bound follows using the ``tree graph'' method of Aizenman \& Newman \cite{AN84}, decomposing $\prob(\vx \leftrightarrow 0, \vy \leftrightarrow 0)$ based on the meeting point of the open paths from $\vx$ to $0$ and from $\vy$ to $0$. See, for instance, Lemma 2.1 from \cite{KN11}.
\end{proof}


If $D \subseteq \Zd$ is a set of vertices, we let $A \sa{D} B$ denote the event that there is an open path of edges, all of whose endpoints are vertices of $D$, connecting $A$ to $B$. 
The cluster of a site $\vx$ restricted to a set $A$ is written
\[C_A(\vx) := \left\{\vy \in A: \vx \stackrel{A}{\longleftrightarrow} \vy \right\}\ . \]


We define the restricted connectivity function $\tau_A(\vx,\vy) = \prob\left(\vx \stackrel{A}{\longleftrightarrow} \vy \right)$. Of special interest is the half-space two-point function invoked in the statement of Theorem \ref{Critical Exponents}, \chJH{written (with some abuse of notation) as} $\tau_H(\vx, \vy) := \tau_{\Zd_+}(\vx,\vy)$. We similarly write $C_H(\vx):= C_{\Zd_+}(\vx)$.

\chJH{The half-space one-arm probability is defined by $\pi_H(n):= \prob(0 \sa{\Zd_+} S'(n))$. There is a possible alternate definition of $\pi_H$: namely, $\prob(0 \sa{\Zd_+} S(n))$, the probability that there is a half-space arm to distance $n$ in the $\ve_1$-direction. We note that the arguments in this paper in fact show that both of these probabilities are asymptotic to $n^{-3}$; see \eqref{eq:armtofarside} and the surrounding discussion below.}

\chJH{We will make reference to the \emph{Harris-FKG} (or ``FKG'') and \emph{BK-Reimer} (or ``BK'') correlation inequalities. We direct the reader to \cite[Chapter 2]{BR06} for statements of, and references to the literature on, these and related inequalities.}





\subsection{Mass-transport\label{sec:masstrans}}
Our proof of the upper bound $\pi_H(n) \leq C n^{-3}$ of Theorem \ref{Critical Exponents} involves considering the point-of-view of a boundary vertex of a spanning cluster of a large box --- that is, the configuration seen from a typical $\vx \in \partial B(n)$ lying in such a spanning cluster. This is made precise by the following lemma, which is an application of the general mass-transport technique. See \cite[Chapter 8]{LP17} for more information about mass-transport.
\begin{lem}
\label{lem:transport}
  Let $h(\vx,\vy)$ be a function from $\Z^d \times \Z^d$ to $[0, \infty]$ which is translation-invariant in the following sense: $h(\vx+\vz, \vy+\vz) = h(\vx,\vy)$ for all $\vx,\,\vy,\,\vz \in \Zd$. Then, for any $\vx \in \Z^d$,
\[\sum_\vz h(0, \vz) = \sum_{\vz}h(\vx,\vz) = \sum_{\vz}h(\vz,\vx) = \sum_\vz h(\vz,0)\ . \]
\end{lem}
\begin{proof}
  Note that $h(0, \vz) = h(0 - \vz, \vz-\vz) =  h(-\vz, 0)$, so
  \begin{align*}
    \sum_{\vz}h(0,\vz) = \sum_{\vz}h(-\vz, 0) = \sum_{\vz} h(\vz, 0)\ .
  \end{align*}
The fact that the value is unchanged when replacing $0$ by $\vx$ follows similarly, again using the translation invariance of $h$.
\end{proof}

Lemma \ref{lem:transport} will be applied to particular \emph{mass-transport rules}. A mass-transport rule is a function $\mass(\cdot, \cdot)$ on $\Zd \times \Zd$ assigning to each pair $\vx,\vy$ a nonnegative random variable $\mass(\vx,\vy) = \mass[\omega](\vx,\vy)$ in a translation-covariant way. In other words, for almost every realization $\omega = (\omega_e)_e$  of the percolation process, we have
\[\mass[\omega](\vx + \vz, \vy + \vz) = \mass[\Theta_\vz \omega](\vx, \vy)\ , \]
where $(\Theta_\vz \omega)_e = \omega_{e +\vz}$ (and addition of a vertex and an edge is defined by $\{\va, \vb\} + \vz = \{\va+\vz, \vb+\vz\}$).

Such an $\mass(\vx,\vy)$ is referred to as the ``mass sent from $\vx$ to $\vy$.'' For a given choice of $\mass$, we apply Lemma \ref{lem:transport} to $h(\vx,\vy) = \E \mass(\vx,\vy)$ (translation invariance of $h$ follows from the translation covariance of $\mass$). In this case, letting $\send = \sum_\vz \mass(0,\vz)$ and $\get = \sum_\vz \mass(\vz,0)$, the lemma states
\[\E \send = \E \get\ . \]

\section{\chJH{Lower bound on $\pi_H(n)$}}\label{sec:pihlb}

Our main goal in this section is to prove \chJH{the lower bound of part (a) of Theorem \ref{Critical Exponents}:}
\begin{prop} \label{hslbarm}
	There is a constant $c=c(d)$ such that 
	$\pi_H(n) \geq c n^{-3}$
	\chJH{for all $n \geq 1$}.
\end{prop}
\chJH{Recall that $Ann(m,n) = B(n) \setminus B(m)$ is the annulus of in-radius $m$ and out-radius $n$.} 
\begin{defin}
	For $r,s\in \dN$ with $r<s$, let 
	\chJH{$ \br{Ann(r,s)}$} 
	be the set of all open clusters \chJH{of} $\Zd$ which \chJH{intersect} both $B(r)$ and $\del B(s)$. 
\end{defin}

The clusters belonging to $\br{Ann(r,s)}$ will be called $Ann(r,s)$-spanning clusters. Note that connectivity in the above definition is determined relative to $\Zd$ and not the annulus; in particular, if $\cC \in \br{Ann(r,s)},$ then $\cC \cap Ann(r,s)$ may be a disconnected set.
We will mostly work with the annulus $Ann(n,3n)$.
For $\cC\in\br{Ann(n,3n)}$, let $X_{\cC}$ denote the number of \chJH{vertices of $\del B(2n) \cap \cC$ which can access $\partial B(n)$ via open paths within $B(2n)$}. More precisely,
\[X_{\cC} := \#\left\{\vx \in \del B(2n) \cap \cC: \vx \sa{B(2n)} B(n) \right\}\ . \]

Next we
define a collection $\sS$ of \chJH{``regular'' annulus} spanning clusters with certain regularity properties. Roughly speaking, \chJH{$\cC \in \sS$} if 
\begin{enumerate}
	\item $X_{\cC}$ is large enough so that $\cC$ \chJH{is likely to extend to the boundary of a} larger ball of radius $O(n)$, say, $B(5n)$. That is, $X_{\cC} \gtrsim n^2$.
	\item $\cC$ contains $\approx n^4$ vertices in boxes of side length $\approx n$.
\end{enumerate}
To be more precise, let $\eta>0$ and
\begin{align*}
\sS_\eta &:= \{\cC\in\br{Ann(n,3n)}: \,   X_{\cC} \geq \eta n^2,
\#\left[\cC \cap Ann(3n,5n)\right] \geq \eta n^4,\,
\#\left[\cC \cap B(5n)\right] \leq \eta^{-1} n^4\}\ .
\end{align*}
Note that $\sS_\eta$ depends on $n$.

\begin{figure}
\centering
(a)\includegraphics[width=7.7cm, height=6cm]{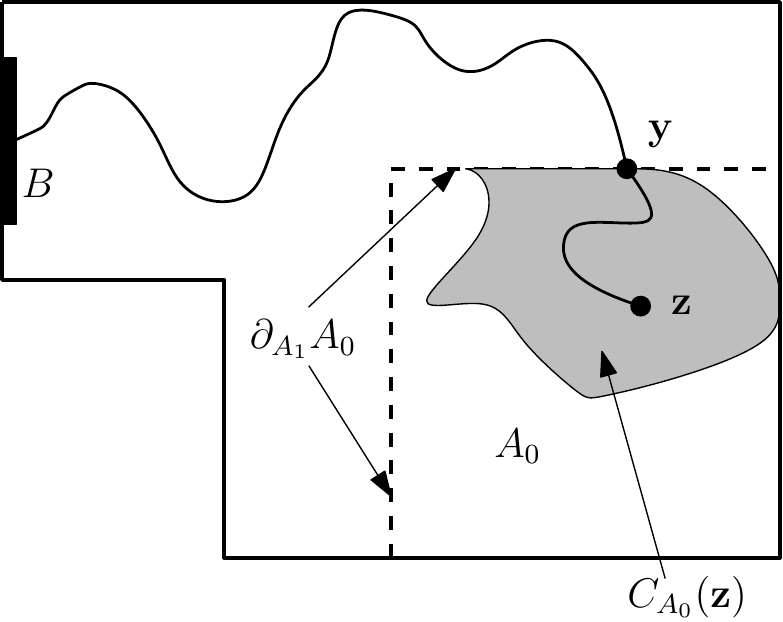}\hfill
(b)\includegraphics[width=7.7cm, height=6cm]{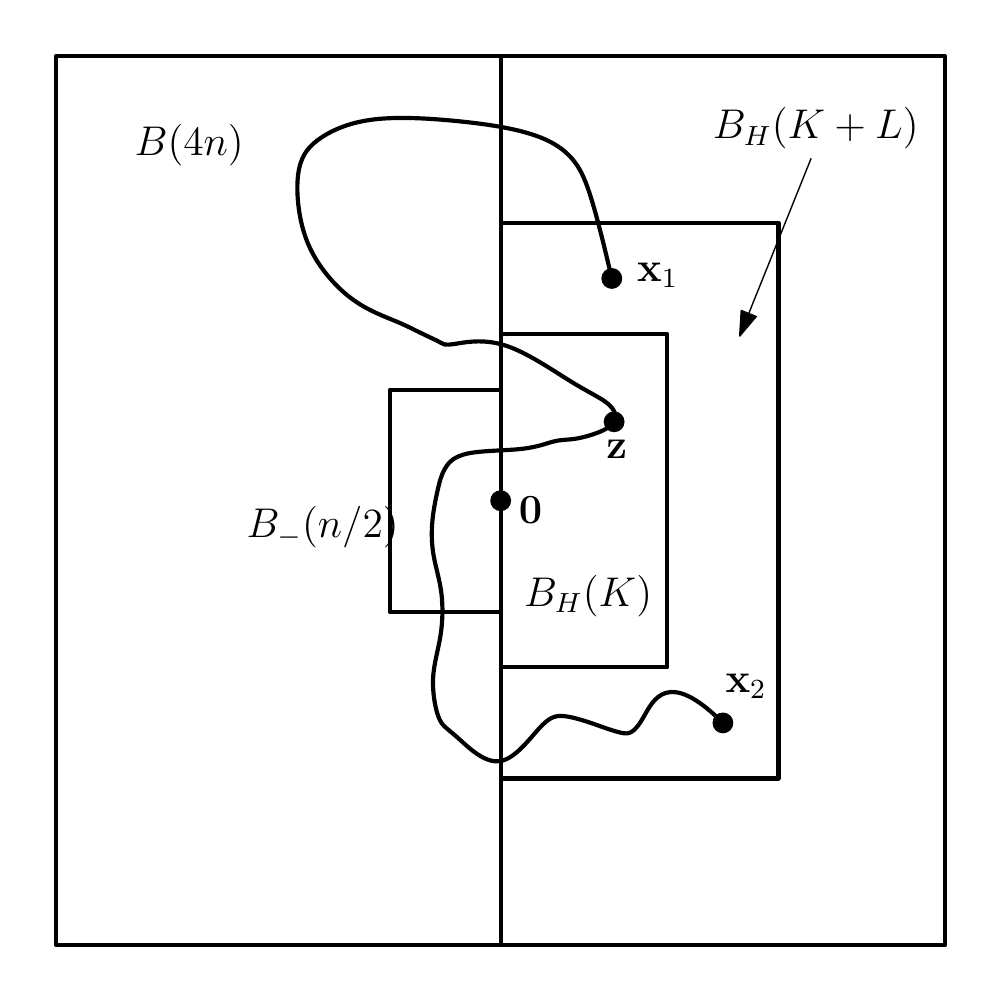}
\caption{Depictions of definitions appearing in Lemma \ref{lem:nofurther} and \eqref{eq:aoutdef} below. (a) This is an instance of the event $\{\vy \lra B\}\circ\{\vy\sa{A_0} \vz\}$ for $\vy\in\del_{A_1}A_0$ as illustrated. (b) All edges are closed except the two paths drawn above, and $\vx_1\in A_\vz^{out}, \vx_2\not\in A_\vz^{out}$.}
\label{Fig: lemma6 and Eqn19}
\end{figure}
\chJH{The following lemma will be useful for showing that $X_{\cC}$ is typically large (and thereby proving the existence of many points with half-space arms). We state it in a general form so that later in the paper it can also be applied to the case of, for instance, nested half-space boxes.}
\begin{lem}
	\label{lem:nofurther}
	\chJH{Let $A_0 \subseteq A_1 \subseteq \Zd$ be arbitrary finite vertex sets with $\vz \in A_0$. Let $B \subseteq \partial A_1$ be a distinguished portion of the boundary of $A_1$, and suppose that the $\ell^\infty$ distance from $A_0$ to $B$ is $\lambda$. Recall the defintion of $\partial_{A_1} A_0$ from \eqref{eq:genbdyv}. Then for all $M > 0$, we have
	\[\prob\left(\vz \stackrel{A_1}{\longleftrightarrow} B \mid C_{A_0}(\vz) \right) \leq  M \pi(\lambda)\]
	almost surely, on the event $\{ \# \{ \vy \in \partial_{A_1} A_0: \, \vz \sa{A_0} \vy \}  = M \}$ (which is measurable with respect to $C_{A_0}(\vz)$).} 
\end{lem}
\begin{proof}
	\chJH{For a vertex set $\cC$ of $A_0$, note that the event $\{C_{A_0}(\vz) = \cC\}$ depends only on the status of edges having either both endpoints in $\cC$ or one endpoint in $\cC$ and one endpoint in $A_0 \setminus \cC$. Conditional on $\{C_{A_0}(\vz) = \cC\}$, if   $\{\vz \sa{A_1} B\}$ occurs, then there must be some $\vy \in \cC \cap \partial_{A_1} A_0$ (see Figure \ref{Fig: lemma6 and Eqn19}(a) for a sketch) such that $\vy \lra B$ off $\cC$. 	That is, $\vy$ has an open path (in $\Zd$) to $B$ which touches $\cC$ only at $\vy$. We thus have the inclusion
	\begin{equation}\label{eq:Jeventdecomp2}
	\{\vz \sa{A_1} B,\, C_{A_0}(\vz) = \cC\} \subseteq  \{ C_{A_0}(\vz) = \cC\} \cap \{ \exists \vy \in \cC \cap \partial_{A_1} A_0 \text{ with }  \vy \lra B \text{ off } \cC\}\ .  \end{equation}
	
 For any fixed $\cC$, the events on the right-hand side of \eqref{eq:Jeventdecomp2} are independent, and the probability that any $\vy \in \cC \cap \partial_{A_1} A_0$ has such a connection is clearly bounded above by $\pi(\lambda)$. Thus,
		\begin{align*}
		\prob\Big(\vz \stackrel{A_1}{\longleftrightarrow} &B , \, C_{A_0}(\vz) = \cC\Big)\\
		&= \prob\left(\vz \stackrel{A_1}{\longleftrightarrow} B \mid \, C_{A_0}(\vz) = \cC \right)  \prob\left( C_{A_0}(\vz) = \cC\right)\\
		&\leq \prob\left(C_{A_0}(\vz) = \cC\right) \prob\left(\text{there is a } \vy \in \cC \cap \partial_{A_1} A_0 \text{ as in } \eqref{eq:Jeventdecomp2} \right)\\
		&\leq M \pi(\lambda) \prob\left(C_{A_0}(\vz) = \cC\right)\ .
		\end{align*}}
\end{proof}

Our main technical work \chJH{in the remainder of this section} is to show the following.
\begin{lem}\label{lem:regcard}
	There exist $\eta_0>0$ and positive constants $c_1=c_1(\eta_0,d)$ and $c_2=c_2(\eta_0,d)$  such that, uniformly in $n$ and $\eta\le\eta_0$,
	\[\prob\left(\# \sS_\eta \geq  c_1\eta n^{d-6} \right) \geq c_2. \]
\end{lem}
We first assume the truth of Lemma \ref{lem:regcard} and use it to prove Proposition \ref{hslbarm}.

\begin{proof}[Proof of Proposition \ref{hslbarm}]
	Let $\eta_0, c_1=c_1(\eta_0,d)$ and $c_2=c_2(\eta_0,d)$ be the constants from Lemma \ref{lem:regcard}.
	First note that if $\cC , \cC'\in\br{Ann(n,3n)}$ are not the same, then the vertices counted in the definition of $X_{\cC}$ and $X_{\cC'}$ are disjoint. So \chJH{for} any $\eta>0$,
	\begin{equation}
	\label{eq:xlbd}
	\#\{\vx\in\del B(2n): \vx\sa{B(2n)} B(n)\} \ge \sum_{\cC \in \sS_\eta} X_{\cC},
	\end{equation}
	and hence  $\#\{\vx\in\del B(2n): \vx\sa{B(2n)} B(n)\} \ge c_1\eta^2 n^{d-4}$ on the event $\{\#\sS_\eta \ge c_1\eta n^{d-6}\}$. In view of Lemma \ref{lem:regcard}, the above event has probability $\ge c_2$ for all $\eta\le \eta_0$. Therefore, for such an $\eta$,
	\beqax  
	\E\#\{\vx\in\del B(2n): \vx\sa{B(2n)} B(n)\} 
	&\ge&  \E\left[\#\{\vx\in\del B(2n): \vx\sa{B(2n)} B(n)\} \mathbf 1_{\{\#\sS_\eta\ge c_1\eta n^{d-6}\}}\right] \\
	&\ge& c_1\eta^2n^{d-4} \prob(\#\sS_\eta\ge c_1\eta n^{d-6}) 
	\ge c_1c_2\eta^2 n^{d-4}.\eeqax
	
	On the other hand, if a vertex $\vx\in\del B(2n)$ satisfies $\vx\sa{B(2n)} B(n)$, then $\vx$ must have a half-space arm to distance $n$ (in fact, a half-space arm ``directed in the $\ve_1$-direction'' --- see \eqref{eq:armtofarside} and the surrounding discussion). So an upper bound for the expectation appearing in the last display is given by
	\begin{equation}
	\label{eq:xubd}
	\#\del B(2n) \cdot \pi_H(n) \leq C_1 n^{d-1} \pi_{H}(n)
	\end{equation}
	for some constant $C_1=C_1(d)$.
	Comparing \eqref{eq:xlbd} to \eqref{eq:xubd} gives $\pi_H(n) \ge (c_1c_2\eta^2/C_1) n^{-3}$, which completes the proof of the proposition.
\end{proof}

We note that a slight extension of the above argument shows something stronger than Proposition \ref{hslbarm}. \chJH{Namely, recalling that $Rect(n) = [0, n] \times [-4n, 4n]^{d-1}$, } there exists a uniform $c > 0$ such that
	\begin{equation}\label{eq:armtofarside}
	\chJH{\prob\left( 0 \stackrel{Rect(n)}{\longleftrightarrow} S(n), \, \#C_{Rect(n)}(0) \cap S(n) > c n^2  \right)} \geq c n^{-3}\ ,
	\end{equation} 
	which shows that the probability of $0$ having a half-space arm directed in the $\ve_1$-direction has the same order as the probability of an undirected half-space arm. We will make use of the strengthened result \eqref{eq:armtofarside} later in the paper (see \eqref{eq:hsxsuff}).

To complete the proof of Proposition \ref{hslbarm}, it suffices to prove Lemma \ref{lem:regcard}. The key fact that we need to prove Lemma \ref{lem:regcard}  is the following.   Recall that for $\vx\in\Zd$, $C(\vx)$ is the open cluster containing $\vx$.
\begin{lem} \label{OneCluster}
	There exists  $\eta_0>0$ and $\cons=\cons(\eta_0,d)>0$ such that for all $\eta\le\eta_0$ and $\vx\in B(n/2)$, $\prob(C(\vx)\in\sS_\eta) \geq \cons n^{-2}$. 
\end{lem}
First, we show how to use Lemma \ref{OneCluster} to prove Lemma \ref{lem:regcard}; we then prove Lemma \ref{OneCluster}. 
\begin{proof}[Proof of Lemma \ref{lem:regcard}]
	We apply  Lemma \ref{OneCluster} to obtain $\eta_0, \cons(\eta_0,d)$ such that
	\begin{equation}
	\label{eq:assumept}
	\prob\left( C(\vx) \in \sS_\eta \right) \geq \cons n^{-2} \text{  uniformly in $n, \eta\le\eta_0$ and $\vx \in B(n/2)$}\ .
	\end{equation}
	Now we will use a  second-moment argument for $\#\sS_\eta$.  First note that 
	\[ \sum_{\vx\in B(n/2)}\mathbf 1_{\{C(\vx)\in\sS_\eta\}} \le\sum_{\cC\in\sS_\eta} \#\cC \cap B(n/2) \le \eta^{-1}n^4 \#\sS_\eta.\]
	The last inequality follows from the fact that $\#\cC\cap B(5n) \le \eta^{-1}n^4$ for all $\cC\in\sS_\eta$. From the last display and \eqref{eq:assumept},
	\beq \label{eq:meambd} \E\#\sS_\eta \ge \frac{\eta}{n^4} \sum_{\vx\in B(n/2)} \prob(C(\vx) \in\sS_\eta) \ge \frac{\eta}{n^4}  C_2 n^d \cons n^{-2} = C_2 \cons\eta n^{d-6}\eeq
	for some constant $C_2=C_2(d)$.
	Now we estimate the second moment of $\#\sS_\eta$. Note that
	\[ \sum_{\vx\in Ann(3n,5n)}\mathbf 1_{\{C(\vx)\in\sS_\eta\}} \ge\sum_{\cC\in\sS_\eta} \#\cC \cap Ann(3n,5n) \ge \eta n^4 \#\sS_\eta.\]
	The last inequality follows from the fact that $\#\cC\cap Ann(3n,5n) \ge \eta n^4$ for all $\cC\in\sS_\eta$. 
	Thus,
	\[ \# \sS_\eta \le \frac{1}{\eta n^{4}} \sum_{\vx\in Ann(3n,5n)} \mathbf 1_{\{C(\vx)\in\sS_\eta\}}, 
	\text{ and so } (\#\sS_\eta)^2  \le \frac{1}{\eta^2 n^{8}} \sum_{\vx,\vy\in 	Ann(3n,5n)} \mathbf 1_{\{C(\vx), C(\vy)\in\sS_\eta\}}. \\
	\]
	For each of the above summands  there are two possibilities based on whether $C(\vx)$ and $C(\vy)$ intersect or not. If $C(\vx), C(\vy) \in \sS_\eta$ and $C(\vx)\cap C(\vy)=\chJH{\varnothing}$, then $\vx$ and $\vy$ are connected to $B(n)$ using disjoint paths, so using the BK inequality
	\beqax 
	\prob(C(\vx), C(\vy)\in \sS_\eta,\, C(\vx)\cap C(\vy)=\chJH{\varnothing}) &\le& \prob(\{\vx\leftrightarrow B(n)\} \circ \{\vy\leftrightarrow B(n)\})  \\
	&\le& \prob(\vx\leftrightarrow B(n)) \prob(\vy\leftrightarrow B(n)).
	\eeqax
	On the other hand, note that for any $\vx\in Ann(3n,5n)$,
	\beqax
	&& \sum_{\vy\in Ann(3n,5n)} \mathbf 1_{\{C(\vx),C(\vy)\in\sS_\eta,\, C(\vx)\cap C(\vy) \ne\varnothing\}}\\
	&\le&  \sum_{\vy\in C(\vx)\cap B(5n)} \mathbf 1_{\{C(\vx)\in\sS_\eta\}} = \#\left(C(\vx)\cap B(5n) \right) \mathbf 1_{\{C(\vx)\in\sS_\eta\}} \le \eta^{-1}n^4 \mathbf 1_{\{C(\vx)\in\sS_\eta\}}
	\eeqax
	by the definition of $\sS_\eta$. Combining the last three displays,
	\beqax 
	\E[(\#\sS_\eta)^2] 
	&\le& \frac{1}{\eta^{2}n^{8}}\left[\sum_{\vx\in Ann(3n,5n)} \eta^{-1}n^4\prob(C(\vx) \in\sS_\eta) 
	+ \sum_{\vx,\vy\in Ann(3n,5n)} \prob(\vx\leftrightarrow B(n))\prob(\vy\leftrightarrow B(n))\right]. 
	\eeqax
	Using \eqref{eq:onearmprob}, we have $\prob(\vx\leftrightarrow B(n)) \le \Cr{c:armhigh} n^{-2}$  uniformly in $\vx\in Ann(3n,5n)$. Since $\#B(5n) = (5n + 1)^d$, the two terms in the right-hand side of the above display are at most $C\eta^{-3} n^{d-6}$ and $C\eta^{-2} n^{2d-12}$ respectively. Therefore, there is a constant $C_3>0$ such that 
	\[ \E[(\#\sS_\eta)^2] \leq C_3\eta^{-3}n^{2d-12}.\]
	Using the estimates in the above display and \eqref{eq:meambd}, and applying \chJH{the Paley-Zygmund inequality},
	\beqax
	\prob(\#\sS_\eta\ge \frac 12C_2\eta\cons n^{d-6})
	&\ge& \prob(\#\sS_\eta\ge\frac{1}{2}\E\#\sS_\eta)\\
	&\ge& \frac{1}{4} \frac{(\E\#\sS_\eta)^2}{\E[(\#\sS_\eta)^2]} \\
	&\ge& \frac{1}{4} \frac{C_2^2\eta^2\cons^2 n^{2d-12}}{C_3\eta^{-3}n^{2d-12}}
	=  \frac{C_2^2\cons^2}{4C_3}\eta^5.
	\eeqax
	While the above bound depends on $\eta$, we can replace it by a constant for $\eta \leq \eta_0$ since the probability appearing in the statement of Lemma \ref{lem:regcard} is decreasing in $\eta$. This completes the proof of the Lemma.
\end{proof}
Lastly, we need to show Lemma \ref{OneCluster}. The lemma follows from moment estimates and Lemma \ref{lem:nofurther}, which says that clusters of boxes with a small number of boundary vertices are likely to die out.
\begin{proof}[Proof of Lemma \ref{OneCluster}]
	Since $\#\sS_\eta$ is monotone in $\eta$,
	it suffices to show that there is a $\eta_0>0$ and $\cons=\cons(\eta_0,d)>0$ such that $\prob(C(\vx)\in\sS_{\eta_0}) \ge \cons n^{-2}$ for all $\vx\in B(n/2)$. The proof consists of the following steps:
	\begin{enumerate}
		\item[Step 1.]  There are  positive  constants $C_1, C_2$ depending on $d$ such that for any $\vx\in B(n/2)$,
		\beq \label{reg est 1}
		\prob\left(\# C(\vx)\cap Ann(3n,5n) > \eta n^4\right) \ge (C_1-C_2\eta^2)n^{-2}\ . \eeq
		\item[Step 2.] There are positive constants $C_1, C_2, C_3$ depending on $d$ such that for any $\vx\in B(n/2)$,
		\beq \label{reg est 2}
		\prob\left(\#C(\vx)\cap Ann(3n,5n) > \eta n^4, \#C(\vx)\cap B(5n) \le \eta^{-1} n^4\right) \ge (C_1 -C_2\eta^2 -C_3\eta) n^{-2}.\eeq
		\item[Step 3.]  There are positive constants $C_1, C_2, C_4$ depending on $d$ such that for any $\vx\in B(n/2)$, 
		\beq \label{reg est 3}
		\prob(C(\vx)\in \sS_\eta) \ge  (C_1 -C_2\eta^2 -C_4\eta) n^{-2}.\eeq
	\end{enumerate} 
	The proof of the lemma follows from Step 3  by taking $\cons(\eta,d):=C_1 -C_2\eta^2 -C_4\eta$ and choosing $\eta_0>0$ small enough so that $\cons(\eta_0,d)>0$. Now we give the proof of the three steps. 
	
	\noindent
	{\bf Step 1}. We will use a second moment argument for the distribution of $\#C(\vx)\cap Ann(3n,5n)$ given $\{\vx\lra \del B(3n)\}$. First note that
	\beqax 
	\E\left(\left. \#C(\vx)\cap Ann(3n,5n)\right| \vx\lra \del B(3n)\right)
	&=& \sum_{\vy\in Ann(3n,5n)} \prob \left(\left. \vx\lra \vy \right| \vx\lra \del B(3n)\right) \\
	&=& \sum_{\vy\in Ann(3n,5n)} \frac{\prob(\vx\lra \vy)}{\prob\left(\vx\lra \del B(3n)\right)},
	\eeqax 
	as $\vx\lra \vy$ implies $\vx\lra\del B(3n)$ for all $\vy\in Ann(3n,5n)$. Equation \eqref{eq:onearmprob} and the symmetries of the lattice give that $\prob(\vx\lra\del B(3n)) \asymp n^{-2}$. This, together with the two point function estimate \eqref{eq:twopt}, gives
	\beq \label{1stmomentbd} 
	\E\left(\left. \#C(\vx)\cap Ann(3n,5n)\right| \vx\lra \del B(3n)\right)
	\ge c_1 n^2 \sum_{\vy\in Ann(3n,5n)} ||\vx-\vy||^{2-d} \ge c_2 n^4\eeq
	for some constants $c_1, c_2$ which depend only on $d$. Next note that
	\beqax 
	\E\left(\left. \left(\#C(\vx)\cap Ann(3n,5n)\right)^2\right| \vx\lra \del B(3n)\right)
	&=& \sum_{\vy, \vz\in Ann(3n,5n)} \prob \left(\left. \vx\lra \vy, \vx\lra \vz \right| \vx\lra \del B(3n)\right) \\
	&=& \sum_{\vy\in Ann(3n,5n)} \frac{\prob( \vx\lra \vy, \vx\lra \vz)}{\prob\left(\vx\lra \del B(3n)\right)},
	\eeqax 
	 as $\vx\lra \vy, \vz$ implies $\vx\lra\del B(3n)$ for all $\vy, \vz\in Ann(3n,5n)$. Now, $\sum_{\vy,\vz \in Ann(3n, 5n)} \prob(\vx \lra \vy,\,\vz)$ is upper-bounded by $\E[(\#C(\vx) \cap [\vx +B(6n)] )^2 ]$, which is at most $c_4 n^6$ for some constant $c_4 > 0$ by Lemma \ref{lem:momentaiz}.  
	
	Combining this estimate with the fact that $\prob(\vx\lra\del B(3n)) \asymp n^{-2}$,
	\beq \label{2ndmomentbd}  
	\E\left(\left. \left(\#C(\vx)\cap Ann(3n,5n)\right)^2\right| \vx\lra \del B(3n)\right) \le c_5 n^8
	\eeq
	for a constant $c_5$ that depends only on $d$. Using the inequalities in \eqref{1stmomentbd} and \eqref{2ndmomentbd}, 
	 applying the Paley-Zygmund inequality, we find
	\beqax 
	&& \prob\left(\left.\# C(\vx)\cap Ann(3n,5n) > \eta n^4\right| \vx\lra \del B(3n)\right) \\
	&\ge& \prob\left(\left.\# C(\vx)\cap Ann(3n,5n) > (\eta/c_2) \E \# C(\vx)\cap Ann(3n,5n)\,\right| \vx\lra \del B(3n)\right) \\
	&\ge& (1-\eta^2/c_2^2) \frac{\left(\E\left[\left.\# C(\vx)\cap Ann(3n,5n) \right| \vx\lra \del B(3n)\right]\right)^2}
	{\E\left(\left[\left.\# C(\vx)\cap Ann(3n,5n) \right| \vx\lra \del B(3n)\right]^2\right)} 
	\ge  (1-\eta^2/c_2^2) c_2^2/c_5.
	\eeqax
	The above estimate together with the fact that $\prob(\vx\lra \del B(3n)) \asymp n^{-2}$ gives \eqref{reg est 1}.
	
	\noindent
	{\bf Step 2}. Combining the first moment bound of Lemma \ref{lem:momentaiz} with the Markov inequality gives  
	\[  \prob\left(\#C(\vx)\cap B(5n) > \eta^{-1} n^4\right) \le c_9\eta n^{-2}.\]
	Using this with the estimate in \eqref{reg est 1}, we get \eqref{reg est 2}.
	
	\noindent
	{\bf Step 3}. It follows from Lemma \ref{lem:nofurther} and the fact that $\prob(0 < X_{C(\vx)}) \leq c n^{-2}$, that
		\[\chJH{\prob(X_{C(\vx)} <\eta n^2, \vx \lra \del B(3n)) \leq \prob(0 < X_{C(\vx)} < \eta n^2) \left[C \eta n^{2} n^{-2}\right] \leq  c_{10} \eta n^{-2}} \]
		uniformly for $\vx \in B(n/2)$.
	Combining the above estimate with \eqref{reg est 2}, we see
	\beqax 
	&& \prob(C(\vx)\in\sS_\eta) \\
	&=& \prob(\# C(\vx)\cap Ann(3n,5n) > \eta n^4, \# C(\vx) \cap B(5n) \le \eta^{-1} n^4)  \\
	&& - \prob(\# C(\vx)\cap Ann(3n,5n) > \eta n^4, \# C(\vx) \cap B(5n) \le \eta^{-1} n^4, X_{C(\vx)}<\eta n^{2}) \\
	&\ge& \prob(\# C(\vx)\cap Ann(3n,5n) > \eta n^4, \# C(\vx) \cap B(5n) \le \eta^{-1} n^4) 
	-  \prob(X_{C(\vx)}<\eta n^{2}, \vx\lra \del B(3n)) \\
	&\ge& (C_1-C_2\eta^2-C_3\eta-c_{10}\eta)n^{-2}=:\cons(\eta)n^{-2}.
	\eeqax
	This shows \eqref{reg est 3}.
\end{proof}

\section{Extending large clusters}\label{sec:extend}

\chJH{In this section, we state several results on regularity and extensibility of clusters. One family of results guarantees that, conditional on a typical realization of (for instance) $C_{B_H(n)}(0)$, the cluster $C(0)$ is not often large.  Another guarantees that when (for instance) $C_{B_H(n)}(0)$ contains enough vertices of $S'(n)$, the cluster $C_{\Zd_+}(0)$ often contains a large number of vertices $\vx$ with $\|\vx\| \approx 2n$. In fact, we will not claim a statement as strong as this in the current section, since it is difficult to rule out the case that $C_{B_H(n)}(0) \cap S'(n)$ is localized near the half-space boundary $S(0)$--- this necessitates working with annular regions. We note that once we show $\pi_H(n) \leq C n^{-3}$, the fact that $C_{B_H(n)}(0)$ does not typically localize near $S(0)$ on $\{0 \lra S'(n)\}$ follows from  \eqref{eq:armtofarside} above.}

In what follows, we will consider a vertex $\vz$ within some connected vertex set $D\subseteq\Zd$ and some subset $Q \subseteq \partial D$ of its boundary. $D$ is generally a box or annulus. We introduce the notation $X_Q(D,\vz)$ for the number of ``boundary vertices'' of $C_D(\vz)$ on $Q$:
\[
  \text{For $\vz\in D\subseteq \Rd$ and $Q\subseteq \del D$, let } X _Q(D,\vz) := \#\{\vx \in Q: \, \vx \sa{D} \vz \} = \#[C_D(\vz) \cap Q]\ .
\]

We first state a ``regularity'' theorem. It says roughly that, if $X_Q$ is large, the clusters of most of the vertices contributing to $X_Q$ are not larger than their typical size. For $s > 0$ and $\vx \in \Zd$ arbitrary, define the event
\[\cT_s(\vx):= \left\{\ \#\left(C(\vx) \cap (\vx + B(s))\right)\  < s^4 \log^7 s \right\}\ . \]
\begin{defin}\label{RegDef}
  Let $D\subseteq \Rd$ and $\vx \in \del D$. For $s > 0$, we say that $\vx$ is $s$-bad with respect to $D$ if
  \[\prob\left(\cT_s(\vx) \mid C_{D}(\vx) \right) \leq 1 - \exp(-\log^2 s)\ . \]
  We say that $\vx$ is $K$-irregular with respect to $D$ if $\vx \in \partial D$ and there is some $s \geq K$ such that $\vx$ is $s$-bad with respect to $D$. Otherwise, $\vx$ is said to be $K$-regular with respect to $D$. We denote the set of $K$-regular vertices of $D$ by $REG_D(K)$.
\end{defin}

We  define the ``irregular version'' of $X_Q(D,\vz)$, which counts the number of boundary vertices whose clusters are abnormally large:
\begin{equation}
\label{eq:defineirreg}
X^{K-irr}_Q(D,\vz) := \#\left\{\vx \in C_D(z) \cap Q:\, \vx \text{ is $K$-irregular with respect to } D\right\}\ . \end{equation}

\chJH{The following lemma provides a tail bound for $X_Q^{K-irr}$ when $X_Q$ is large, for a growing sequence of annuli or boxes $D$. Suppose that for each $n$, the set $D$ is a dilation of the same box or annulus --- that is, if $D$ is a translate of $\prod_{i=1}^d [\alpha_i n, \beta_i n]$, or the annuli $Ann(cn ,n)$, $\Anns(cn, n)$, or $Ann_H(cn, n)$, where the $\alpha_i$'s, $\beta_i$'s, or $c$ are fixed. We say $Q$ is a dilated subrectangle of $\partial D$ for each $n$ if $Q$ is a $(d-1)$-dimensional rectangle in $\partial D$ with nondegenerate sides and if, for each $n$, $Q$ is dilated and translated as $D$ is --- i.e., as $n$ increases, $Q$ changes by the same dilations / translations as $D$. }

\begin{lem}[Cluster regularity]
\label{thm:regthm}
\chJH{Consider a sequence of growing (in $n$) domains $D$ which are dilations / translations of the same box or annulus having sidelength order $n$ as in the above paragraph. Suppose that $Q$ is a dilated subrectangle of $\partial D$, also as above. There exist constants $C > c > 0$  and $K_0>0$ such that for any $n, M$ and any $K\ge K_0$, the following holds. Uniformly in $\vz\in D$,
we have
  \[\prob\left(X _Q(D,\vz)\geq M \text{ and } X^{K-irr}_Q(D,\vz) \geq \frac 12 X_Q(D,\vz) \right) \leq C n^d \exp(- c \log^2 M)\ .\]}
\end{lem}

A version of Lemma \ref{thm:regthm} in the case that $D$ is a cube $B(n)$ and  $Q  = \partial B(n)$ was proved as Theorem 4 of \cite{KN11}. Lemma \ref{thm:regthm} follows by an argument similar to the proof of that result; we omit the details.  The main use of Lemma \ref{thm:regthm} will be in ``extensibility'' arguments allowing the enlargement of the cluster of a site $\vx$, conditional on the value of $C_D(\vx)$.


Such extensibility arguments were also a key part of the argument showing \eqref{eq:onearmprob} appearing in \cite{KN11}. The set-up we will use differs from these previous extensibility results in a major way: namely, we typically want to extend clusters restricted to lie in the subgraph $\Zd_+$. This poses a couple of serious obstacles. The first problem is that we cannot use the usual two-point function $\tau(\vx,\vy)$ for lower bounds on the probability that long open connections exist, since $\tau(\vx,\vy)$ includes contributions from the event where such connections leave $\Zd_+$. More precisely, we need to compare $\prob\left(\vx \stackrel{\Z^d_+}{\longleftrightarrow} \vy \right)$ to $\tau(\vx,\vy)$. A main aim of Theorem \ref{thm:boxcon} is to provide a comparison between these two connectivity probabilities when $\vx$ and $\vy$ are a macroscopic distance from $S(0)$.

The second problem relates to our inability to effectively localize the half-space arm from $0$ on the event $\left\{0 \stackrel{\Z^d_+}{\longleftrightarrow} S'(n)\right\}$. Ideally, we would prove $\pi_H(2n) \geq c \pi_H(n)$ by conditioning on the existence of an arm to distance $n$ and showing it is likely to be extended.  This would require one to show that the distance-$n$ arm does not typically terminate close to $S(0)$, since the two-point function in $\Z^d_+$ behaves very differently near $S(0)$ than far from $S(0)$. Proving that half-space arms can be localized away from the boundary appears to be difficult a priori; to solve this problem we work in an annulus $\Anns$ and compare to the case of the half-space. As mentioned above, such a localization result does ultimately follow as a consequence of $\pi_H(n) \leq C n^{-3}$ and \eqref{eq:armtofarside}; this will be important for our work on the two-point function in (b) of Theorem \ref{Critical Exponents}.

For simplicity, we introduce the following abbreviations for \chJH{stating} the extensibility result. If $\vz \in B_H(k)$, where $n \leq k \leq 2n$ and if $0 < L \leq 3n - k$ is an integer, we define (see Figure \ref{Fig: lemma6 and Eqn19}(b) for a sketch)
\begin{equation}
\label{eq:aoutdef}
\mathsf{A}^{out}_\vz(n,k, L) :=  \left[ C_{\Anns(n/2,4n)}(\vz) \cap Ann_H(k, k+L)  \right]\ .
\end{equation}
If $\vz \in B_H(4n)$, we define
\[\mathsf{A}^{in}_\vz(n) := \left[ C(\vz) \cap B_{-}(n/4) \right]\ . \]
In this language, the main theorem on extensibility is as follows:

\begin{thm}\label{thm:mainextend}
There is some constant \chJH{$c_* > 0$} such that the following hold uniformly in $n \geq c_*^{-1}$, in $n^{1/10} \leq L \leq 3n - k$, in $n \leq k \leq 2n$ and in $M$ and $\vz$ as specified. 

\begin{itemize}
\item Let $D = B_H(k)$ and $Q = S'(k)$. Uniformly in $M \geq L^2 / 2$,
\begin{equation}
\label{eq:outextend}
\prob\left( \#\mathsf{A}^{out}_0(n,k,L) \leq c_* M L^2,\, X_Q(D, 0) = M \right) \leq (1-c_*) \prob(X_Q(D,0) = M).
\end{equation}
\item Let $D = Rect(n)$ and $Q =\partial_{\Zd_+} Rect(n)$ (the union of sides of $Rect(n)$ not lying along $S(0)$). Uniformly in $M \geq n^2 / 2$,
\begin{equation}
\label{eq:rectextend}
\prob\left( \#\mathsf{A}^{out}_0(4n,4n,8n) \leq c_* M n^2,\, X_Q(D, 0) = M \right) \leq (1-c_*) \prob(X_Q(D,0) = M).
\end{equation}
\item Let $D = \Anns(n/4,5n) $ and $Q = \partial_{-}\Anns(n/4,5n).$  Uniformly in $\vz \in \chJH{B_H(4n)}$ and in $M \geq n^2 / 2$,
\begin{equation}
\label{eq:inextend}
\prob\left(\#\mathsf{A}^{in}_\vz (n) \leq c_* M n^2,\, X_Q(D,\vz) = M \right) \leq (1-c_* ) \prob(\chJH{X_Q(D,\vz)} = M).
\end{equation}
\end{itemize}
\end{thm}
\chJH{We defer the proof of Theorem \ref{thm:mainextend} to Section \ref{sec:proveext}. We first, in Section \ref{sec:rtwopt}, prove Theorem \ref{thm:boxcon}, since it will be used in the proof of Theorem \ref{thm:mainextend} to generate open paths avoiding $B_{-}(n/4)$.}

\section{Proof of Theorem \ref{thm:boxcon}\label{sec:rtwopt}}
Note that the upper bound claimed in the theorem follows from the unrestricted two-point function: $\tau_{D}(0,\vx) \leq \tau(0,\vx) \leq \Cr{c:twopthigh}\|\vx\|^{2-d}$ for any $D \subseteq \Zd$. We will first  give the matching lower bound in a more restrictive setting than claimed in the theorem. The restriction will be removed via an inductive argument that bootstraps a lower bound on the two-point function $\tau_{B(n)}$ far from the box boundary to one slightly closer to the box boundary.

We now state the ``restrictive setting'' version of Theorem \ref{thm:boxcon} alluded to above.

\begin{prop}\label{prop:weakboxcon}
  There exist constants $M_0 > 1$ and $\cons_1 > 0$ such that the following holds uniformly in $n$. For all $\vx \in B(n) \setminus \{0\}$,
\[\tau_{B(M_0 n)}(0,\vx) \geq  \cons_1 \|\vx\|^{2-d}\ . \]
\end{prop}

\begin{proof}
\chJH{We say $\vx \lra \vy \text{ through } D$ if $\vx \lra \vy$ but every open path from $\vx$ to $\vy$ uses a vertex of $D$.}
Suppose $\vx \in B(n)$. Note that for any $M>1$, the event $\{0\lra \vx\}$ is a disjoint union of $\{0\sa{B(Mn)} \vx\}$ and $\{0\lra \vx \text{ through } B(Mn)^c\}$. Thus,

\begin{align*}
\tau_{B(Mn)}(0,\vx) = \prob(0 \leftrightarrow \vx) - \prob(0\lra \vx \text{ through } B(Mn)^c)\ .
\end{align*}
By \cite[(1.12)]{HS14}, the latter term of the right-hand side is bounded above by $C(Mn)^{2-d}$, uniformly in $\vx \in B(n)$. Using \eqref{eq:twopt}, the first term of the above is at least $\Cr{c:twoptlow} \|\vx\|^{2-d}$. Choosing $M$ large completes the proof.
\end{proof}

The result of Proposition \ref{prop:weakboxcon} will serve as the base case for an induction argument, which will prove Theorem \ref{thm:boxcon}. \chJH{In fact, our argument shows that the nested cubes of that theorem can be replaced by possibly oblong rectangles of arbitrary fixed aspect ratio. We state this strengthened version of the theorem for future reference:
\begin{thm}
	\label{thm:newboxtwopt}
	Fix $\alpha_i,\, \beta_i > 0$ for $1 \leq i \leq d$; fix also $M > 1$. For each $n$, let the rectangle
	\[R_n := [-\alpha_1 n, \beta_1 n] \times \ldots \times [-\alpha_d n, \beta_d n]\ . \]
	There is some $c = c(M, (\alpha_i), (\beta_i))$ such that, uniformly in $n$ and in $\vx \in R_n$,
	\[ \tau_{M R_n}(0,\vx) \geq c \|\vx\|^{2-d}\ . \]
	\chJH{Here, by $M R_n$, we mean the dilation of $R_n$ considered as a subset of $\bbR^d$, not as a subset of $\Z^d$ --- i.e., the set of $\vy \in \Zd$ such that $y(i) \in [-\alpha_i M n, \beta_i M n]$ for all $i$.}
\end{thm}}
	For use in the proof, we introduce some shorthand for the boundary vertices of cubes reachable from $0$ within the cube:
\[ X^{box}(n):= X_{\partial B(n)}(B(n),0) =   \#\{\vx\in\partial B(n): 0 \sa{B(n)} \vx\}\ ,\]
\chJH{where in the first equality we use the notation of Section \ref{sec:extend} with $D = B(n)$ and $Q = \partial B(n)$.  }
We need a lemma bounding $\E X^{box}(n)$ for our proof of Theorem \ref{thm:boxcon}.
\begin{lem}[Theorem 1.5(a) of \cite{HS14}]\label{EXnbd}
	There is a constant $\Cons_1>0$ such that $\E X^{box}(n) \le \Cons_1$ uniformly in $n \geq 1$. 
\end{lem}

\noindent
{\it Proof of Theorem \ref{thm:boxcon} and Theorem \ref{thm:newboxtwopt}. }
  \chJH{We prove the notationally simpler case of a cube --- that is, we prove Theorem \ref{thm:boxcon} --- in detail, then describe the modifications necessary for other rectangular regions.} Let $F^R(\cdot):=||\cdot||^{d-2}\tau_{B(R)}(0, \cdot)$. For $M > 1$, say
  that  
  $\tau$ is  {\it $M$-good} if  there are constants $c(M), n_0(M)$ so that $F^{Mn}|_{B(n)}\ge c$ for all $n\ge n_0$.  The proof of Theorem \ref{thm:boxcon} is inductive, and Proposition \ref{prop:weakboxcon} initializes the induction. 
The inductive \chJH{step is accomplished by the following claim}.
\begin{clam} \label{ind hyp}
If $\tau$ is $M$-good and $\alpha(M) := \min\{4/3, (M + 1)/2\}$, then $\tau$ is  $(M/\alpha(M))$-good.
\end{clam}
 It is not hard  to see that
if $\tau$ is $M_0$-good for some $M_0>1$ (which is guaranteed  by Proposition  \ref{prop:weakboxcon}),  then one can show that $\tau$ must be $M$-good for any $M\in (1, M_0)$ by applying Claim \ref{ind hyp} finitely many times.
This proves Theorem \ref{thm:boxcon}.  
 
 To prove Claim \ref{ind hyp} it is enough to show that if $F^{Mn}|_{B(n)}$ is bounded away from 0, then so is $F^{Mn}|_{B(\alpha(M)n)}$. So, if 
$B_j(n) := \{\vx\in \Zd: |x(1)|, \ldots, |x(j)|\le \alpha(M)n;\, |x(j+1)|, \ldots, |x(d)|\le n\}$~obey
 \begin{clam} \label{ind hyp 1}
If $F^{Mn}|_{B_j(n)}$ (where $0 \leq j < d$) is bounded away from 0 for all $n$ large enough, then so is  
$F^{Mn}|_{B_{j+1}(n)}$.
\end{clam}
  \noindent
 then Claim \ref{ind hyp} follows from Claim \ref{ind hyp 1} by using induction on $j$.
Note that the hypothesis of Claim \ref{ind hyp} initializes the 
 induction argument for Claim \ref{ind hyp 1} at $j=0$.

To show Claim \ref{ind hyp 1} suppose $F^{Mn}|_{B_j(n)}$ is bounded away from 0 for some  $0\leq j<d$, so for some constant $c_M>0$,
\beq \label{assume}
\tau_{B(Mn)}(0,\vx)\geq c_M||\vx||^{2-d} \;\ \text{ for all }n \geq 1 \text{ and } \vx\in B_j(n).
\eeq
 Fix an arbitrary $\vx \in B_{j+1}(n)\setminus B_j(n)$. We will bound $\tau_{B(Mn)}(0,\vx)$ from below. Without loss of generality we can assume  that $x(i) \geq 0$ for all $i$, as other cases are similar. Let
 \beq  \label{D prop}
D = \vx + B((\alpha(M)-1)n), \text { so } D \subseteq B(Mn)\setminus \chJH{B(n/3) }
 \eeq
 by our choice of $\alpha(\cdot)$. Also,
  \chJH{$\partial D$} contains the $(d-1)$-dimensional quadrant
  \[ Q:=\{\vy\in D: y(i)\leq x(i) \chJH{\text{ for all $i \neq j+1$, and } y(j+1) = x(j+1)} -\lfloor (\alpha(M)-1)n\rfloor\}\ . \] 
  \chJH{Either $Q \subseteq B_j(n)$, or for each vertex $\vz$ of $Q$, we have $\vz - \ve_{j+1} \in B_j(n)$.  }
 
\begin{figure}
\centering
(a)\includegraphics[width=7.7cm, height=6cm]{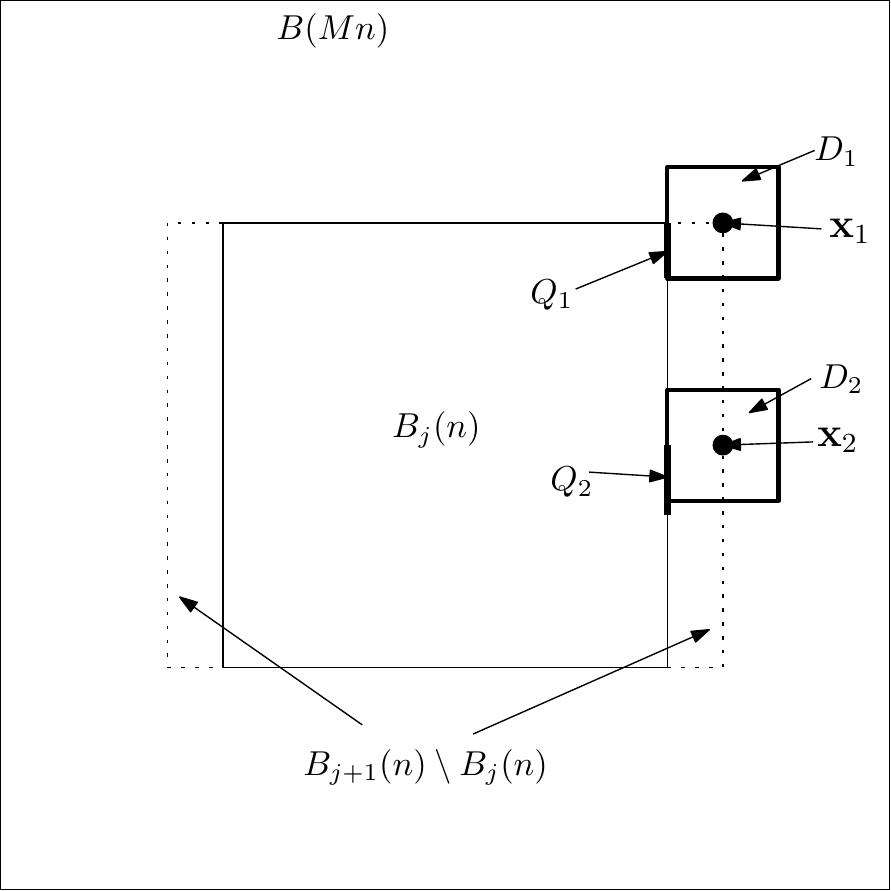}\hfill
(b)\includegraphics[width=7.7cm, height=6cm]{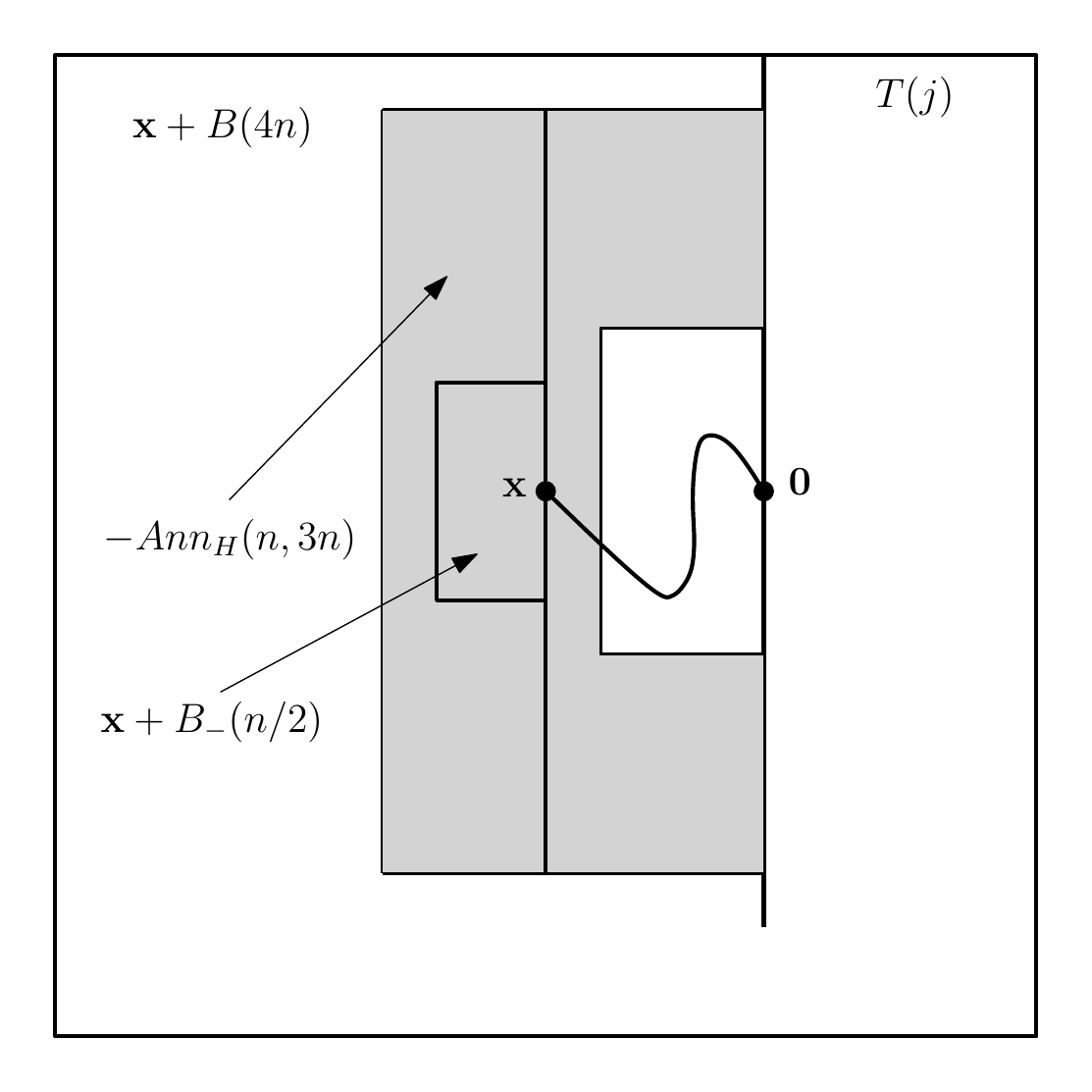}
\caption{(a) Referenced above \eqref{ind hyp suff}. $(\vx_i, D_i, Q_i), i=1, 2,$ are two possible locations of $(\vx, D, Q)$. (b) Referenced below \eqref{eq:justtofig}. This is an instance of the event when $\vx\in T(j)$ has an open connection to $0$ staying within $\vx+B(4n)$ and avoiding $\vx+B_-(n/2)$.}
\label{Fig: Dloc}
\end{figure}

If $\vx$ is on the $i$-th axis for some $i$, then an entire side of $D$ (perpendicular to the $i$-th axis) containing $Q$ lies in, \chJH{or is shifted by $\ve_{j+1}$ from,} $B_j(n)$. At the other extreme, when $\vx$ is at the corner of $B_{j+1}(n)$ belonging to $\{\vy\in\Zd: y(i)\geq 0\}$, then either $\del D\cap B_j(n)$ \chJH{or $[\partial D - \ve_{j+1}] \cap B_j(n)$} equals $Q$.  See Figure \ref{Fig: Dloc}(a)  for possible locations of $D$.
 Now note that if  $F_\vz:=\{\vz\sa{D} \vx ,\, \vz\sa{B(Mn)} 0\}$, then 
Claim \ref{ind hyp 1} will follow if we show that there is a constant $c>0$ (independent of $\vx$ and $n$) such that
\beq \label{ind hyp suff}
\prob(\cup_{\vz\in Q} F_\vz) \geq c n^{2-d} \text{ for all $n$ large enough,}
\eeq
because $\cup_{\vz\in Q} F_\vz$ implies $\{\vx\sa{B(Mn)} 0\}$. To prove \eqref{ind hyp suff}, let 
$Y_Q^K$ be the number of $\vz$ in $Q\cap C_D(\vx)\cap REG_D(K)$ such that $\vz \sa{B(Mn)}0$ and
such that  the edge $\{\vz, \vz'\}$ is pivotal for the event $\{\vx \leftrightarrow 0\}$,
 where $\vz' = \vz - \ve_{j+1}$ lies in $B_j(n) \setminus D$.
The following lemma gives bounds for the (conditional) moments of $Y_Q^K$. As above, we introduce abbreviated notation for $X_Q(D,\vx)$ in order to make equations more readable.
\begin{lem}\label{lem:yrestfirst}
(1) Let $X_{\del D} := \#\del D\cap C_D(\vx), X_Q := \#Q\cap C_D(\vx)$ and
 $X_Q^{K-reg}:=\# Q\cap C_D(\vx)\cap REG_D(K)$.
There are constants $\eta,\cons_1(\eta)>0$ (independent of $\vx$ and $n$) such that 
\beq    \label{eq:quadbd}
\text{if } B_\eta:=\{\eta n^2 <  X_Q^{K-reg} \leq X_Q \leq X_{\del D} < \eta^{-1} n^2\}, \text{ then }
  \prob(B_\eta) \geq \cons_1 n^{-2}\ .
\eeq
(2) Let $\eta>0$ be such that \eqref{eq:quadbd} holds. There are constants $K_0, \Cons_2, \cons_2>0$ such that for all $K>K_0$ and all $\eta n^2 < N < \eta^{-1}n^2$,
\beqax
(2A) && \E [(Y_Q^K)^2; X_Q^{K-reg} = N, B_\eta] \leq \Cons_2 n^{4-d} \prob\left(X_Q^{K-reg} = N;  B_\eta\right)\ ;\\
(2B) && \E [Y_Q^K; X_Q^{K-reg} = N, B_\eta] \geq \cons_2 n^{4-d} \prob\left(X_Q^{K-reg} = N; B_\eta\right)\ . 
\eeqax
\end{lem}
\noindent
Using Lemma \ref{lem:yrestfirst} and the second-moment method,  if $K>K_0$ then
\begin{align*}
 \prob(Y_Q^K > 0 \mid X_Q^{K-reg} = N, B_\eta) & \geq \frac{\cons_2^2}{\Cons_2} n^{4-d} \;\forall \; N\in(\eta n^2, \eta^{-1} n^2), \text{ which implies}\\
\prob(Y_Q^K > 0 ) & \geq\sum_{\eta n^2 < N <\eta^{-1} n^2} \prob(Y_Q^K > 0 \mid X_Q^{K-reg} = N, B_\eta)\prob(X_Q^{K-reg} = N, B_\eta)  \\
& \geq \frac{\cons_2^2}{\Cons_2} n^{4-d} \prob(B_\eta) \geq \frac{\cons_1\cons_2^2}{\Cons_2} n^{2-d} \text{ using \eqref{eq:quadbd}}.
\end{align*}
 This proves \eqref{ind hyp suff}, as $\{Y_Q^K>0\}$ implies $\cup_{\vz\in Q} F_\vz$, and thus completes the proof of Claim \ref{ind hyp 1}.
\hfill \qed

We end the section by proving Lemma \ref{lem:yrestfirst}.

\begin{proof}[Proof of Lemma \ref{lem:yrestfirst}]
{\bf 1. } From the definition of $Q$ \chJH{and the symmetries of the lattice} it is not hard to see that $\#C_D(\vx)\cap\del D$ is bounded above by a sum of $d2^{d}$ copies of $X_Q$ which are identically distributed (but not independent). So, using a union bound and Lemma \ref{OneCluster}, there are constants $\eta_0(d)>0$ and $\cons(\eta_0,d)>0$ such that
\[ \prob(X_Q>2\eta n^2)\geq \frac{1}{d2^d}\prob(\#C_D(\vx)\cap \del D>d2^{d+1} \eta n^2) \geq \frac{\cons}{d2^d}n^{-2} \text{ for all } \eta\leq \eta_0.\]
Also, Lemma \ref{thm:regthm} \chJH{implies}
\[ \prob(X_Q>2\eta n^2, X^{K-reg}_Q\leq \eta n^2) \leq Cn^d\exp(-c\log^2(2\eta n^2)) \text{ for some constants $C, c>0$.}\]
Finally, using Lemma \ref{EXnbd} and \chJH{the} Markov inequality, $\prob(X_{\del D}\geq \eta^{-1} n^2) \leq \Cons_1\eta n^{-2}$.
Combining this with the last two displays,
\[ \prob(B_\eta) \geq \frac{\cons}{d2^d}n^{-2} - Cn^d\exp(-c\log^2(2\eta n^2)) - \Cons_1\eta n^{-2} \text{ for all } \eta\leq \eta_0.\]
So we get the desired result if we choose $\eta>0$ small enough and $n$ large enough.

\noindent {\bf (2A). }  First we argue that $Y^K_Q \leq 1$ a.s.~via the method of contradiction. Suppose, if possible, $\vz_1$ and $\vz_2$ are two vertices counted in $Y^K_Q$. Then $\vx \leftrightarrow 0$, so we can choose a self-avoiding open path  $\gamma$ joining $\vx$ to $0$. By pivotality, $\gamma$ must contain the edges $\{\vz_i, \vz'_i\}$ for $i=1,\,2$. Suppose (without loss of generality) that $\gamma$ passes through $\vz_2$ first when traversed from 0 to $\vx$.  Then we can find a path $\gc'\subseteq\gc$ joining 0 and $\vz_2$ such that the edge $\{\vz_1, \vz_1'\} \not\in\gc'$. On the other hand, since $\vz_2\in C_D(\vx)$, we also have a path $\gc''$ which stays entirely within $D$ and joins $\vx$ and $\vz_2$. This contradicts the fact that the edge $\{\vz_1, \vz'_1\}$ is pivotal for $\{\vx\lra 0\}$, as $\gc'\cup\gc''$ avoids the edge $\{\vz_1, \vz_1'\}$ and connects $\vx$ and 0.
Thus $Y^K_Q\leq 1$. In particular,\\
$(Y^K_Q)^2 \chJH{=} \sum_\vz \mathbf 1_{\{\vz \text{ counted in } Y_Q\}}$. 
Conditioning on the cluster of $\vx$, $\E[(Y^K_Q)^2; X_Q^{K-reg} = N, B_\eta]$ is 
\begin{align*}
&\leq \sum_{\cC \in B_\eta \cap \{X_Q^{K-reg} = N\} } \E[(Y_Q^K)^2; C_D(\vx) = \cC] \\
&\leq   \sum_{\cC \in B_\eta \cap \{X_Q^{K-reg} = N\} }\  \sum_{\substack{\vz\in Q: \,\vz \in \chJH{\cC \cap }REG_D(K)\\\text{when } C_D(\vx) =\cC}}\prob(C_D(\vx) = \cC,\, \vz \lra 0 \text{ off } \cC)
\end{align*}
\chJH{(recall that ``$\vz \lra 0$ off $\cC$'' means that there is an open path from $\vz$ to $0$ touching $\cC$ only at $\vz$).} Using \eqref{eq:twopt} and the fact that $Q\cap B(n/3)=\varnothing$, along with the independence of the above events, we see as in the proof of Lemma \ref{lem:nofurther} that the above is bounded by
\begin{align*}
&\Cr{c:twopthigh} (n/3)^{2-d} \sum_{\vz\in Q} \prob(\vz \in C_D(\vx)\cap REG_D(K), X_Q^{K-reg} = N; B_\eta)\\
&= \Cr{c:twopthigh} (n/3)^{2-d} \E X_Q^{K-reg}\mathbf 1_{\{X_Q^{K-reg} = N\}\cap B_\eta} \leq \Cr{c:twopthigh} (n/3)^{2-d}\eta^{-1}n^2 \prob(X_Q^{K-reg}=N; B_\eta).
\end{align*}
This completes the proof of (2A) of Lemma \ref{lem:yrestfirst}.

\noindent
{\bf (2B). } We will define some events that force $Y_Q^K$ to be nonzero. For $\vz \in Q$, consider the box
$\wtD_\vz = \vz - [K/2, K]^d$.
 Since $x(i)\geq 0$, $\widetilde D_\vz \subseteq B_j(n)$ when $n > K \chJH{\geq 2}$. Also note that $\widetilde D_\vz$ is at distance $K/2$ from $D$. In what follows, for a fixed $\vz\in Q$, $\tz$ will typically denote a vertex of $\widetilde D_\vz$; $N$ will also always be a value between $\eta n^2$ and $\eta^{-1} n^2$.
 Define the events
\begin{align*}
\cE_1(\vz,N) := B_\eta\cap\left\{\vx \stackrel{D}{\leftrightarrow} \vz,\, \vz \in REG_D(K), \text{ and } \chJH{X_Q^{K-reg}} = N\right\} \\
\cE_2(\vz,\tz):=\left\{\tz\sa{B(Mn)} 0 \text{ off } C_D(\vz)\right\}, \quad \cE_3(\vz,\tz) := \left\{C(\vz) \cap C(\tz) = \varnothing \right\}\ .
\end{align*}
We successively bound probabilities of the intersections of the $\cE_i$'s via a series of claims.
\begin{clam}\label{clam:e1}
 Let $c_M$ be the constant from \eqref{assume}. There is a   constant $K_0 \chJH{\geq 2}$ (depending on $c_M$) such that $\prob\left(\cE_1(\vz, N) \cap \cE_2(\vz, \tz) \right) \geq (c_M/2) n^{2-d}\prob(\cE_1(\vz,N))$ for all \chJH{$\vx$}, $K>K_0,\, n\geq 10K, \, \vz \in Q,\, \tz \in \wtD_\vz$ and $N\geq 1$.
\end{clam}
\begin{figure}
\centering
(a)\includegraphics[width=7.8cm, height=6cm]{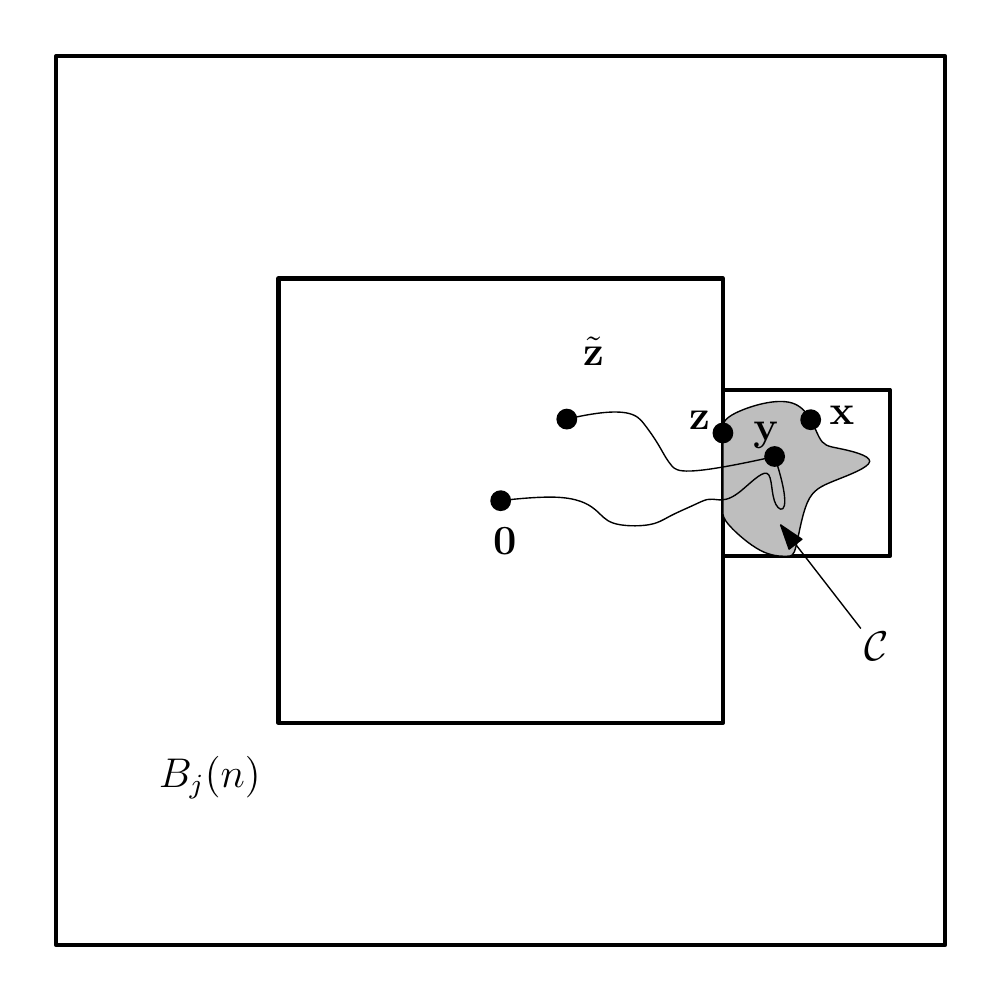}\hfill
(b)\includegraphics[width=7.8cm, height=6cm]{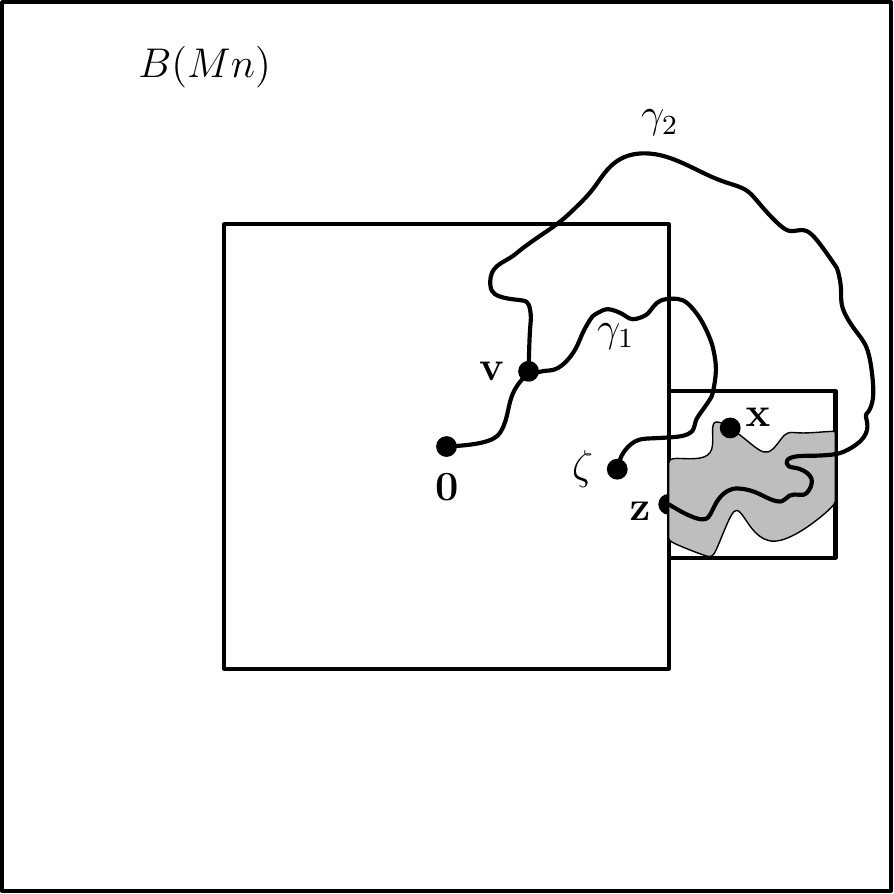}
\caption{(a) The event $\{\tz \lra \vy \} \circ \{\vy \lra 0\}$ for $\vy \in \cC$ as in the proof of Claim \ref{clam:e1}. The shaded region represents $C_D(\vx)$. (b) Bounding a cluster intersection event in the proof of Claim \ref{clam:e2}. Depicted is the event  $\{\vx \leftrightarrow \vv; \cE_1(\vz,N)\} \circ \{\zeta~\leftrightarrow~\vv\} \circ \{\vv \leftrightarrow 0\}$.}
\label{Fig: Eqn27 and 30}
\end{figure}
Note that for any realization $\cC$ of $C_D(\vz)$ satisfying $\cE_1(\vz,N)$,  $\prob(\cE_2(\vz,\tz) \mid C_D(\vz)=\cC)$ equals
 \beq\label{E2bd}
 \prob(\tz \stackrel{B(Mn)}{\longleftrightarrow} 0 \text{ off } \cC)
 \chJH{\geq} \prob(\tz \sa{B(Mn)} 0)  - \prob(\cup_{\vy\in\cC}\{\tz \lra \vy\}\circ \{\vy\lra 0\}). \eeq
 See Figure \ref{Fig: Eqn27 and 30}(a) for a sketch. Using  \eqref{assume} and recalling that $\tz\in B_j(n)$, the first term in the RHS of \eqref{E2bd} is $\geq  c_M n^{2-d}$. Using a union bound and the BK inequality, \eqref{eq:twopt} and the fact that $\cC\subseteq \chJH{(B(n/3))^c}$ (see  \eqref{D prop}), the second term in the RHS of \eqref{E2bd}  is  $\leq \Cr{c:twopthigh} (n/3)^{2-d} \sum_{\vy \in \cC} \prob(\tz \lra \vy)$. From \eqref{E2bd} and the last two observations,
 \beq\label{E2bd1}
 \prob(\cE_2(\vz,\tz)\mid C_D(\vz)=\cC) \geq c_Mn^{2-d}-\Cr{c:twopthigh}(n/3)^{2-d} \sum_{\vy\in \cC} \prob(\tz\lra \vy).\eeq
In order to estimate the sum in \eqref{E2bd1}, let
 $\chJH{U_r} :=\tz+Ann(2^r, 2^{r+1})$ for $r\geq 0$. So $\prob(\tz\lra \vy) \leq \Cr{c:twopthigh} 2^{r(2-d)}$ for all $\vy\in U_r$, which gives  $\sum_{\vy \in \cC} \prob(\tz \lra \vy) \leq \sum_{r\geq\log_2(K/2)}  \Cr{c:twopthigh} 2^{r(2-d)}(\#\cC\cap U_r)$. Since $\|\vz-\tz\| \leq K$, we have $U_r\subseteq \vz+B(2^{r+2})$ for all $r\geq\log_2(K/2)$. Hence, whenever $\cC$ satisfies $\cE_1(\vz,N)$, we have
 \begin{align*}
 \#\cC \cap U_r &\leq \E[\#C(\vz)\cap(\vz+B(2^{r+2}))\mid C_D(\vz)=\cC] \\
 &\leq 2^{4(r+2)}\log^7(2^{r+2}) + 2^{(r+4)d}\prob(\chJH{\cT_{2^{r+2}}(\vz)^c} \mid C_D(\vz) = \cC)\\
 &\leq C 2^{4r} \log^{7}(2^r) 
 \end{align*}
for all $r\geq\log_2(K/2)$, where $C$ is independent of $r$ and $K$ (as long as $K$ is large). In the above, we have used the definition of $K$-regularity and Lemma \ref{thm:regthm}.
 This implies $\sum_{\vy\in\cC} \prob(\tz\lra \vy)$\linebreak $\leq c_1\sum_{r\geq\log_2(K/2)}(r^72^{r(6-d)}) \leq c_2K^{6-d}\log^7K$ for some constants $c_1,\,c_2$ (independent of $K$ and $n$). Using this bound and \eqref{E2bd1}, we see that if $K$ is large enough then
 $\prob(\cE_2(\vz,\tz)\mid C_D(\vz)=\cC) \geq (c_M/2)n^{2-d}\mathbf 1_{\{\cC\in\cE_1(\vz,N)\}}$. 
 Taking an expectation over $C_D(\vz)$ completes the proof of Claim \ref{clam:e1}. 
 
 Having proved Claim \ref{clam:e1}, we move on to the next subsidiary claim, which deals with  $\cE_3$.

\begin{clam} \label{clam:e2}
Let $c_M$ be the constant from \eqref{assume}. There is a constant $K_1 > K_0$ (depending on $c_M$) such that for all $\chJH{\vx}$, $K\geq K_1, n\geq 10K$ and $\vz \in Q$, we can find a $\tz \in \widetilde D_\vz$ satisfying
$\prob(\cE_1(\vz,N) \cap \cE_2(\vz,\tz) \cap \cE_3(\vz,\tz)) \geq (c_M/4) n^{2-d} \prob(\cE_1(\vz,N))$.
\end{clam}
Claim \ref{clam:e2} will follow if we show
that there is a constant $K_1 > K_0$ such that, for any $z\in Q$, if $\zeta$ denotes a uniformly chosen random vertex in $\wtD_\vz$ and if $E_\zeta$ denotes expectation over $\zeta$, 
then 
\begin{equation} \label{eq:zetaavg}
E_\zeta \prob(\cE_1(\vz,N) \cap \cE_2(\vz,\zeta) \cap \cE_3(\vz,\zeta)) \geq (c_M/4)  n^{2-d}\prob(\cE_1(\vz,N)) \text{ for all  $N$ and $K\geq K_1$}. 
\end{equation}
Fix $\vz\in Q$ and $\zeta\in\wtD_z$. Consider the event $(\cE_1(\vz,N) \cap \cE_2(\vz,\zeta)) \setminus \cE_3(\vz,\zeta)$. On this event, 
we can find a self-avoiding open path $\gc_1$ joining $\zeta$ and 0 and avoiding $C_D(\vz)$, then subsequently find a path $\gc_2$ starting at $\vz$ and terminating at its first and only intersection point with $\gc_1$. So if $\vv\in\gc_1\cap\gc_2$  is the \chJH{unique such intersection point of $\gc_1$ and $\gc_2$}, then the event
$\{\vx \leftrightarrow \vv; \cE_1(\vz,N)\} \circ \{\zeta \leftrightarrow \vv\} \circ \{\vv \leftrightarrow 0\}$
occurs (see Figure \ref{Fig: Eqn27 and 30}(b) for a sketch). So, using the union bound, the BK inequality, \eqref{eq:twopt}  and the convention $0^{2-d}=1$,
\begin{align}
 &\prob\left((\cE_1(\vz,N) \cap \cE_2(\vz,\zeta)) \setminus \cE_3(\vz,\zeta)\right) 
  \leq \Cr{c:twopthigh}^2\sum_{\vv \in B(n/100)} \prob\left(\vx \leftrightarrow \vv; \cE_1(\vz,N) \right)\, \|\zeta - \vv\|^{2-d} \,   \|\vv\|^{2-d} \nonumber\\
  & \quad  + \Cr{c:twopthigh}^2\sum_{\vv \notin B(n/100)} \prob\left(\vx \leftrightarrow \vv; \cE_1(\vz,N) \right)\, \|\zeta - \vv\|^{2-d} \,   \|\vv\|^{2-d}=: I_1+I_2\ .
\label{eq:sumvforzeta}
\end{align}
We bound $E_\zeta I_1$ and $E_\zeta I_2$ uniformly in $K$ large, and in $n$ large relative to $K$. First consider $I_1$. If $n \geq 10K$, then using the triangle inequality $\|\zeta - \vv\| \geq \|\vz\|-\|\vz-\zeta\|-\|\vv\|\geq n/2$ for each $\vv\in B(n/100)$.  
Also, 
\[ \prob\left(\vx \leftrightarrow \vv; \cE_1(\vz,N) \right) = \sum_{\cC\in\cE_1(\vz,N)} \prob\left( C_D(\vx) = \cC\right) \prob\left(\vx\leftrightarrow \vv \mid C_D(\vx) = \cC \right).\]
If $\{\vx \leftrightarrow \vv\}$ occurs, then there must be some $\vw\in C_D(\vx)\cap \del D$ such that $\{\vw \leftrightarrow \vv \text{ off } C_D(\vx)\}$ occurs. In particular, using \eqref{eq:twopt} and the fact that $\|\vv-\vw\|\geq  n/2$ for  all $\vw\in \partial D$,
\begin{align*}
  \prob\left(\vx\leftrightarrow \vv \mid C_D(\vx) = \cC \right)  \leq \sum_{\vw\in \cC\cap \del D} \prob(\vw \leftrightarrow \vv)\leq \Cr{c:twopthigh}(n/2)^{2-d} X_{\del D} \leq \Cr{c:twopthigh} \eta^{-1} n^2 (n/2)^{2-d}
\end{align*}
for all $\cC$ satisfying $\cE_1(\vz,N)$.
Pulling the above bounds together and summing over $\cC$ and $\vv$,
\beq
I_1 \leq c_1 \prob(\cE_1(\vz,N)) n^{6-2d} \sum_{\vv \in B(n/100)} \|\vv\|^{2-d}
\leq c_2 \prob(\cE_1(\vz,N)) n^{8-2d} 
\eeq
uniformly in $\zeta$, for some constants $c_1,\,c_2$ (independent of $K$ and $n$).

To control $I_2$, we bound $\|\vv\|^{2-d}$ uniformly by $(n/100)^{2-d}$.  Define $\cC_{\zeta,t} := C(\vx) \cap[ \zeta +  Ann(2^{t-1},2^t)]$ for $t\geq 0$ and $t_K := \log_2 (4K)$.
Since $\|\zeta - \vv\| \geq 2^{t-1}$ when $\vv \in \cC_{\zeta,t}$, 
\begin{align}
  I_2 & \leq C (n/100)^{2-d} \sum_{t=0}^\infty \sum_{\vv \in \zeta + Ann(2^{t-1},2^t)} \prob(\vx\leftrightarrow \vv; \cE_1(\vz,N)) \|\zeta - \vv\|^{2-d}\label{eq:almostsame}\\ &\leq C \left(\sum_{t\geq t_K}\,  2^{2t-dt} \E[\#\cC_{\zeta,t};\, \cE_1(\vz,N)] \right.
  \left.+  \sum_{\substack{t< t_K \\\vv \in \zeta + Ann(2^{t-1},2^t)}} \prob(\vx\leftrightarrow \vv; \cE_1(\vz,N)) \|\zeta - \vv\|^{2-d} \right) n^{2-d} \nonumber \\
  &=: I_{21}+I_{22}
\nonumber
\end{align}
for some constant $C>0$. To bound $I_{21}$ note that $\cC_{\zeta,t}\subseteq \vz+B(2^{t+1})$ for all $t\ge t_K$ and $\zeta\in \wtD_\vz$, so using Lemma \ref{thm:regthm} and discarding a negligible contribution from the event \chJH{$\cT_{2^{t+1}}(\vz)^c$} as before, there is a constant $C$ independent of $n$ and (sufficiently large) $K$ such that
\beqax 
\E[\#\cC_{\zeta,t}; \cE_1(\vz,N)] 
&=& \sum_{\cC\in\cE_1(\vz,N)}\E[\#C(\vz)\cap(\vz+B(2^{t+1}))\mid C_D(\vz)=\cC] \prob(C_D(\vz)=\cC) \\
&\leq&  C \prob(\cE_1(\vz,N)) 2^{4t}\log^7(2^{t}) \ ,  \text{ which implies} \\
I_{21} &\leq& C_3 \chJH{\prob(\cE_1(\vz,N))} n^{2-d}\sum_{t\geq t_K} t^72^{t(6-d)} \leq C_4 \chJH{\prob(\cE_1(\vz,N))}  n^{2-d}K^{6-d}\log^7K 
\eeqax
where the $C_i$'s are constants independent of \chJH{$\vx$}, of $K$ sufficiently large, of $\chJH{\vz}$ and $\zeta$, and of $n$ large relative to $K$.
 We turn now to estimating $E_\zeta(I_{22})$. Consider the expectation $E_\zeta$ of the inner sum for a typical value of $t \leq t_K$. 
 \begin{align*}
 &E_\zeta\sum_{\vv\in\chJH{\zeta + Ann(2^{t-1},2^t)}}\prob(\vx\lra \vv;\cE_1(\vz,N))\|\zeta-\vv\|^{2-d}
 =\chJH{\left[\#\wtD_\vz \right]^{-1}}\sum_{\zeta\in\wtD_\vz}\sum_{\vv\in\chJH{\zeta + Ann(2^{t-1},2^t)}}\prob(\vx\lra \vv;\cE_1(\vz,N))\|\zeta-\vv\|^{2-d} \\
& \leq \chJH{ CK^{-d}}\sum_{\vv\in \cup_\zeta \chJH{\zeta + Ann(2^{t-1},2^t)}}  \prob(\vx\lra \vv;\cE_1(\vz,N))  \left[ \sum_{\zeta\in\wtD_\vz: |\zeta-\vv|>K} K^{2-d}
+  \sum_{l=1}^K\sum_{\zeta\in\wtD_\vz: \|\zeta-\vv\|_{\infty}	=l}
 l^{2-d}\right] \\
 & \leq \chJH{C K^{-d}}\sum_{\vv\in \cup_\zeta \chJH{\zeta + Ann(2^{t-1},2^t)}}  \prob(\vx\lra \vv;\cE_1(\vz,N))  
 \left[(K/2)^dK^{2-d} + \sum_{l=1}^K 2dl^{d-1}l^{2-d}\right] \\
 & \leq C_5 K^{2-d}\E\left[\chJH{\#\left(\cup_{\zeta\in \wtD_\vz} \cC_{\zeta,t}\right)};\cE_1(\vz,N)\right] \text{ for some constant $C_5$}.
 \end{align*}
 Note that $\cC_{\zeta,t}\subseteq \vz+B(5K)$ for all $t\leq t_K$ and $\zeta\in\wtD_\vz$, as $\|\zeta-\vz\|\leq K$. Therefore, the~above~is
 \begin{align}
   &\leq C_5 K^{2-d}\E[\#C(\vz) \cap(\vz+B(5K)); \cE_1(\vz,N)]\nonumber\\
   &= C_5 K^{2-d}\sum_{\cC\in\cE_1(\vz,N)}\prob(C_D(\vz)=\cC)\nonumber \E[\#C(\vz) \cap(\vz+B(5K))\mid C_D(\vz)=\cC]\\  
  &\leq C_5 K^{2-d}(5K)^4\log^7(5K) \prob(\cE_1(\vz,N)) + C_6 K^{2-d} K^{d} e^{-t_K^2} \prob(\cE_1(\vz,N))\ , \label{eq:randomizedvert}
 \end{align}
 again using $K$-regularity.
 
 The second term of \eqref{eq:randomizedvert} is negligible, which implies $E_\zeta(I_{22})  \leq \allowbreak C_7 K^{6-d}\log^8(5K) n^{2-d} \prob(\cE_1(\vz,N))$
  for some constant \chJH{$C_7$}.
 Inserting our estimates for $I_1, I_{21}$ and $E_\zeta(I_{22})$ in  \eqref{eq:sumvforzeta}, we bound $E_\zeta \prob([\cE_1(\vz,N) \cap \cE_2(\vz,\zeta)]\allowbreak \setminus\allowbreak \cE_3(\vz,\zeta))$. Using this bound, the LHS of \eqref{eq:zetaavg}  is \chJH{at least
\begin{equation}\label{eq:touselaterman} E_\zeta\left[ \prob(\cE_1(\vz,N) \cap \cE_2(\vz,\zeta))  \right] - C_8 n^{2-d}K^{6-d}\log^8(5K)\allowbreak\prob(\cE_1(\vz,N))\end{equation}} for some constant $C_8$.
Choosing $K$ large enough and applying Claim \ref{clam:e1},  \eqref{eq:zetaavg} is established.  This finishes the proof of Claim \ref{clam:e2}.

We now move to complete the proof of (2B) of Lemma \ref{lem:yrestfirst}. Suppose we have a pair $(\vz, \tz)$, where $\vz\in Q$ and $\tz\in\wtD_\vz$, as in Claim \ref{clam:e2}. We claim that there is a constant $c_9 = c_9(K)>0$ such that
\begin{equation}\label{eq:modmapped}
  \prob\left(\vz \text{ is counted in } Y^K_Q;\, X_Q^{K-reg} = N, B_\eta \right) \geq \chJH{c_9} \prob(\cE_1(\vz,N)\cap \cap_{i=2}^3 \cE_i(\vz,\tz))\ .\end{equation}
  
 The argument for \eqref{eq:modmapped} is a usual edge modification argument, which we now sketch. We define a function $\Upsilon$ mapping each edge configuration $\omega \in \cE_1(\vz,N)\cap \cap_{i=2}^3 \cE_i(\vz,\tz)$ to a new edge configuration $\Upsilon(\omega)$ as follows. Consider such an outcome $\omega$, with $\tz$ chosen as in Claim \ref{clam:e2}. We can choose according to some deterministic search algorithm a path $\pi$ of open edges from $\tz$ to $0$ lying entirely in $B(Mn)$. Since $C(\vz)$ and $C(\tz)$ are disjoint, this path is guaranteed not to intersect $C(\vz)$. Now, we close all edges having an endpoint in the box $[\vz + B(4K)] \setminus D$, except those edges belonging to $\pi$; we then open $\{\vz, \vz'\}$. Last, we open one-by-one the edges in a path from $\vz'$ to $\pi$ which lies entirely in $[\vz + B(3K)] \setminus [\vx + B((\alpha(M)-1)n + 1)]$ (i.e., the set $D$ widened by one unit) except for its initial vertex $\vz'$.

 It is easy to see that the above procedure connects $\vz$ to $0$ within $B(Mn)$ but that every open path from $\vz$ to $0$ must pass through $\vz'$. Because, in the outcome $\omega$, $\vz$ was in $C_D(\vz) \cap REG_D(K)$ and $B_\eta \cap \{X_Q^{K-reg} = N\}$ initally occurred, and since no edges of $D$ were modified by $\Upsilon$, these facts still hold true for $\Upsilon(\omega)$. Lastly, since the function is at most $e^{CK^d}$-to-one, the probability of the image $\Upsilon(\cE_1 \cap \cE_2 \cap \cE_3)$ is at least $c(K) \prob(\cE_1 \cap \cE_2 \cap \cE_3)$.
 
 Given \eqref{eq:modmapped}, the conclusion of the proof is immediate. Summing \eqref{eq:modmapped} over $\vz$, we find
 \[ \E[Y_Q^K; X_Q^{K-reg} = N; B_\eta] \geq c  \sum_{\vz} \prob(\cE_1(\vz,N)\cap \cap_{i=2}^3 \cE_i(\vz,\tz))\ . \]
 Using Claim \ref{clam:e2}, the probability appearing on the right-hand side is at least $c n^{2-d} \prob(\cE_1(\vz, N))$ when $\tz$ is chosen appropriately.
 Now, on $B_\eta \cap \{X_{Q}^{K-reg} = N\}$, there are $N$ vertices $\vz$ such that $\cE_1(\vz, N)$ occurs; since $N > \eta n^2$, this completes the proof.

\end{proof}

\section{Proof of Theorem \ref{thm:mainextend}}\label{sec:proveext}
 
We will prove only \eqref{eq:outextend}, since \eqref{eq:rectextend} has a very similar proof, and since both \eqref{eq:outextend} and \eqref{eq:rectextend} are harder than \eqref{eq:inextend} (involving, in particular, the restricted cluster appearing in $\mathsf{A}^{out}$). For the purpose of abbreviation, we write $X(m)$ for $X_{S'(m)}(B_H(m),0)$ throughout this section only, and similarly set $X^{K-irr}(m) = X_{S'(m)}^{K-irr}(B_H(m),0)$ and $REG(K,m) = REG_{B_H(m)}(K)$ (recall Definition \ref{RegDef}).

Although some parts of the  arguments here are similar to that of 
  Section \ref{sec:rtwopt}, there are many differences in the details. We will need to build extensions of spanning clusters of large boxes, involving a number of parameters. The statements that follow will provide various bounds that are uniform in $n$ sufficiently large with $n \leq k \leq 2n$, $n^{1/10} \leq L \leq 3n-k$, and $M \geq L^2/2$. The main restriction on $n$ will come from it having to be very large relative to the regularity parameter $K$, which will be fixed relative to all other parameters but larger than some dimension-dependent constant.

We say a pair of vertices $(\vz,\vy)$ is $(k,L, K)$-admissable if
\begin{enumerate}
\item $\vz \in S'(k)$ and $\vy \in (\vz + B_H(L)) \setminus B_H(k)$;
\item $\vz \in REG(K,k)$;
\item $0 \stackrel{B_H(k)}{\longleftrightarrow} \vz$;
\item $\vz \stackrel{\Anns(n/2,4n)}{\longleftrightarrow} \vy$;
\item The status of the edge $\{\vz, \vz'\}$ is pivotal for the event $0 \leftrightarrow \vy$, where $\vz'$ is a deterministically chosen neighbor of $\vz$ in $[\vz+B_H(K)] \setminus B_H(k)$.
\end{enumerate}

Define the random number of admissable pairs
\[Y(k, L, K) = \#\left\{(\vz,\vy): \, (\vz,\vy) \text{ is $(k,L,K)$-admissable} \right\} \ . \]
Let $X^{K-reg}(k)  = X(k) - X^{K-irr}(k) = \# REG(K,k)$.
The argument will follow from the second-moment method, using the bounds in the following pair of lemmas, followed by a local modification argument similar to that in the proof of Lemma \ref{lem:yrestfirst}.
\begin{lem}
\label{thm:mombd1}
 Let $K$ be fixed larger than some dimension-dependent constant. There exists a constant $c = c(K) > 0$ such that
\begin{align}
\E Y(k, L, K) \indi_{X^{K-reg}(k) = M} \geq c M L^2 \prob(X^{K-reg}(k) = M)\ ,
\label{eq:mombd1}
\end{align}
uniformly in $n$ large (relative to $K$), for $n \leq k \leq 2n$, $n^{1/10} \leq L \leq 3n-k$, and $M \geq L^2 / 2$.
\end{lem}

\begin{lem}
\label{thm:mombd2}
Let $K$ be fixed larger than some dimension-dependent constant. There exists a constant $C = C(K)$ such that the following holds for all $n$ large, for $n \leq k \leq 2n$, $n^{1/10} \leq L \leq 3n-k$, and $M \geq L^2 / 2$:
\begin{equation}\label{eq:mombd2}
  \E Y(k, L, K)^2 \indi_{X^{K-reg}(k) = M} \leq C M^2 L^4 \prob(X^{K-reg}(k) = M)\ .
\end{equation}

\end{lem}

\begin{proof}[Proof of Lemma \ref{thm:mombd1}]
  As in the proof of (2B) from Lemma \ref{lem:yrestfirst}, we introduce three events which can be used to build connections from $\vz$ to $\vy$. In these definitions, we generally have $\vz \in S'(k)$, $\vy \in [\vz + B_H(L)] \setminus B_{H}(k),$ and $\tz \in (\vz + B_H(2K)) \setminus B_{H}(k+K)$. 
 Let
 \begin{align*}
   \cE_1(\vz, K, M) &:= \left\{ \vz \stackrel{B_H(k)}{\longleftrightarrow} 0,\, \vz \in REG(K,k),\text{ and } X^{K-reg}(k) = M \right\}\ ;\\
\cE_2(\vz, \tz, \vy) &:= \left\{\tz \stackrel{\Anns(n/2,4n)}{\longleftrightarrow} \vy \text { off } C_{B_H(k)}(\vz) \right\}\ ;\\
\cE_3(\vz, \tz) &:= \left\{C(\vz) \cap C(\tz) = \varnothing \right\}\ .
 \end{align*}

We continue by proving a pair of claims about the probabilities of these events.
\begin{clam}\label{clam:twoE}
  There exists a $c > 0$ depending only on $d$ such that the following holds. Let $K$ be larger than some fixed dimension-dependent constant, and $n$ be large relative to $K$; let $n^{1/10} \leq L \leq 3n-k$ and $M \geq L^2 / 2$. For any $\vz \in S'(k)$ and $\tz \in (\vz + B_H(2K))\setminus B_H(k+K)$, we have
\[\sum_{\vy \in [\vz + B_H(L)] \setminus B_{H}(k)} \prob\left(\cE_1(\vz, K, M) \cap \cE_2(\vz, \tz, \vy)\right) \geq c L^2 \prob(\cE_1(\vz, K, M))\ . \]
\end{clam}
\begin{proof}
  Note that the status of $\cE_1$ can be determined by examining $C_{B_H(k)}(\vz)$. We can thus condition on $C_{B_H(k)}(\vz)$ and bound the conditional probability of $\cE_2$, similarly to the beginning of the proof of Claim~\ref{clam:e1}:
\begin{align*}
 \sum_{\vy \in [\vz + B_H(L)] \setminus B_H(k)} &\prob\left(\cE_1(\vz, K, M) \cap \cE_2(\vz, \tz, y)\right) \\
&\geq  \sum_{\substack{\cC \in \cE_1 }} \prob(C_{B_H(k)}(\vz) = \cC) \sum_{\vy \in [\vz + B_H(L)] \setminus B_H(k) } \prob(\tz \stackrel{\Anns(n/2,4n)}{\longleftrightarrow} \vy \text{ off } \cC \mid C_{B_H(k)}(\vz) = \cC)\\
&=  \sum_{\substack{\cC \in \cE_1 }} \prob(C_{B_H(k)}(\vz) = \cC) \sum_{\vy \in [\vz + B_H(L)] \setminus B_H(k)} \prob(\tz \stackrel{\Anns(n/2,4n)}{\longleftrightarrow} \vy \text{ off } \cC)\ ,
\end{align*}
where we have used the fact that the events in the last sum depend on disjoint sets of edges. We estimate the terms of the second sum using a union bound on vertices of $\cC$:
\begin{align*}
  \prob(\tz \stackrel{\Anns(n/2,4n)}{\longleftrightarrow} \vy \text{ off } \cC) &\geq \prob(\tz \stackrel{\Anns(n/2,4n)}{\longleftrightarrow} \vy) - \sum_{\zeta \in \cC}\prob\left( \{\zeta \leftrightarrow \tz\} \circ\{ \zeta \leftrightarrow \vy \}\right)\ ,
\end{align*}
where we have used the fact that $\{\tz \leftrightarrow \zeta\} \supseteq \{\tz \stackrel{\Anns(n/2,4n)}{\longleftrightarrow} \zeta\}$ and similarly with $\{\zeta \leftrightarrow \vy\}$.
Applying the BK inequality gives the bound
\begin{align*}
 \prob(\tz \stackrel{\Anns(n/2,4n)}{\longleftrightarrow} \vy) - \sum_{\zeta \in \cC} &\prob\left(\{\zeta \leftrightarrow \tz\} \circ \{\zeta \leftrightarrow \vy\} \right) \geq \prob(\tz \stackrel{\Anns(n/2,4n)}{\longleftrightarrow} \vy) - \sum_{\zeta \in \cC} \prob\left(\zeta \leftrightarrow \tz)\prob(\zeta \leftrightarrow \vy \right)\\
&\geq \prob(\tz \stackrel{\Anns(n/2,4n)}{\longleftrightarrow} \vy) - \sum_{t=\lfloor \log_2(K) \rfloor}^\infty \sum_{\substack{\zeta \in \cC \\
\zeta \in [\tz + Ann(2^t, 2^{t+1})]}}   \prob\left(\zeta \leftrightarrow \tz)\prob(\zeta \leftrightarrow \vy \right)\ .
\end{align*}
Note we began the sum above not from $t = 0$ because $\tz$ is at least distance $K$ away from $\cC$. 

We sum the above over $\vy$ and use Theorem \ref{thm:boxcon} on the first term on the right-hand side, finding a lower bound of $c L^2$ for a $c$ uniform for parameter values as in the statement of Claim \ref{clam:twoE}. (Our restrictions on the value of $n$ and $L$ force $L$ to be large relative to $K$ so that the distance between $\tz$ and the ``typical'' $\vy$ is order $L$.)
For the other term, we use \eqref{eq:twopt} for an upper bound on the two-point function; the result is
\[\sum_{\vy \in [\vz + B_H(L)] \setminus B_H(k)} \prob(\tz \leftrightarrow \vy \text{ off } \cC) \geq cL^2 -  CL^2 \sum_{t=\lfloor\log_2(K) \rfloor}^\infty \sum_{\substack{\zeta \in \cC \\
\zeta \in [\tz + Ann(2^t, 2^{t+1}])}}   \prob(\zeta \leftrightarrow \tz)\ .\]

Furthermore, we have $\tz + B(2^s) \subseteq \vz + B(2^{s+1})$ for $s \geq \log_2(2K)$, and note that for any $\cC$ satisfying the requirements of $\cE_1$ and any $m \geq K$, we necessarily have $\#(\cC  \cap \vz + B(m)) \leq m^4 \log^7(m)$. Using these in the above gives a lower bound of
\begin{align*}
  &\geq c L^2 - C L^2 \sum_{t=\log_2(K)}^\infty (\#(\cC \cap [\vz + B(2^{t+2})])) 2^{t(2-d)}\\
  &\geq c L^2 - C' L^2 \sum_{t=\log_2(K)}^\infty t^7 2^{4t} 2^{t(2-d)}\\
  &\geq c L^2 - C'' L^2 K^{6-d} \log^7(K) \ .
\end{align*}
Again, the constant $C''$ is uniform for parameter values in the appropriate range. Therefore, whenever $K$ is sufficiently large and fixed relative to the other parameters, the second term is negligible relative to the first.
\end{proof}

Our next claim gives the ability to add on $\cE_3$ to the intersection in the last claim. 

\begin{clam}\label{clam:threeE}
  For each $K > 0$ sufficiently large (larger than some dimension-dependent constant), there exists a $c  = c(K) > 0$  such that the following holds uniformly in $n$, $k$, $L$, and $M$ as in the statement of Theorem \ref{thm:mombd1}. For any $\vz \in S'(k)$, there exists a $\tz \in [\vz + B_H(2K)] \setminus B_{H}(k+K)$ such that
\[\sum_{\vy \in [\vz + B_H(L)] \setminus B_{H}(k)} \prob(\cE_1(\vz, K, M) \cap \cE_2(\vz, \tz, \vy) \cap \cE_3(\vz, \tz)) \geq c L^2 \prob(\cE_1(\vz, K, M))\ . \]
\end{clam}
\begin{proof}
  Let $\zeta$ be a uniformly chosen (independently of the percolation process) random vertex of $[\vz + B_H(2K)] \setminus B_{H}(k+K)$, and let $E_{\zeta}$ denote expectation with respect to this random choice. We will prove that for $K$ large,
\begin{equation}
\label{eq:unifx}
E_\zeta\sum_{\vy \in [\vz+B_H(L)]\setminus B_{H}(k)} \prob(\cE_1 \cap \cE_2 \cap \cE_3) \geq c L^2 \prob(\cE_1)\ ,
 \end{equation}
where $\cE_2 = \cE_2(\vz, \zeta, y)$ and $\cE_3 = \cE_3(\vz, \zeta)$. This will suffice to show the claim. Indeed, for \eqref{eq:unifx} to hold, there must be some $\tz$ such that, when $\zeta = \tz$, the quantity inside the expectation $E_\zeta$ is at least $c L^2 \prob( \cE_1)$.

For any possible value of $\zeta$, if $\cE_1 \cap \cE_2 \cap \cE_3^c$ occurs, then there exists a vertex $\vv$ such that
\[\cE_1 \cap \{ 0 \leftrightarrow \vv\} \circ \{\zeta \leftrightarrow \vv\} \circ \{\vv \leftrightarrow \vy\} \]
occurs. (Compare to the reasoning above \eqref{eq:sumvforzeta}, where a similar vertex $v$ is found.) In particular, by the BK inequality, for this value of $\zeta$ we have
\[ \prob\left(\cE_1 \cap \cE_2 \cap \cE_3^c\right) \leq \sum_{\vv \in \Z^d} \prob\left(\cE_1 \cap \{0 \leftrightarrow \vv\}\right) \prob\left(\zeta \leftrightarrow \vv\right) \prob\left(\vv \leftrightarrow \vy \right)\ .\]

Summing the above over $\vy \in [\vz+B_H(L)] \setminus B_{H}(k)$ and using \eqref{eq:twopt}, we get a factor of at most a constant multiple of $L^2$, uniform in the value of $\zeta$. Applying \eqref{eq:twopt} again:
\begin{align}
\nonumber
  \sum_{\vy \in [\vz+B_H(L)]\setminus B_H(k)}\prob\left(\cE_1 \cap \cE_2 \cap \cE_3^c\right) &\leq C L^2\sum_{\vv \in \Z^d} \prob\left(\cE_1 \cap \{0 \leftrightarrow \vv\} \right) \prob\left(\zeta \leftrightarrow \vv\right)\\
  &\leq C' L^2\sum_{\vv \in \Z^d} \prob\left(\cE_1 \cap \{0 \leftrightarrow \vv\} \right) \|\zeta - \vv\|^{2-d} \ .   \label{eq:sumrand}
\end{align}
The right-hand side of \eqref{eq:sumrand} is nearly identical to that of \eqref{eq:almostsame}. The differences are that now $0$ plays the role of $\vx$, there is a different prefactor ($C' L^2$ instead of $C n^{2-d}$), and the definition of $\cE_1$ is somewhat modified. A proof very similar to the one used to treat \eqref{eq:almostsame} gives that (compare to the negative term in \eqref{eq:touselaterman})
\[E_{\zeta} \sum_{\vv \in \Z^d} \prob\left(\cE_1 \cap \{0 \leftrightarrow \vv\} \right) \|\zeta - \vv\|^{2-d} \leq C'' K^{6-d} \log^8(K) \prob(\cE_1)\ , \]
uniformly over $K$ sufficiently large and over $n$, $k$, $L$, $M$, and $\vz$ as in the statement of Claim \ref{clam:threeE}. 

We can thus uniformly lower-bound $E_\zeta \prob(\cE_1 \cap \cE_2 \cap \cE_3)$:

\begin{align}
  \chJH{E_{\zeta}} \sum_{\vy \in [\vz + \chJH{B_H(L)}] \setminus \chJH{B_{H}(k)}}\prob(\chJH{\cE_1 \cap \cE_2 \cap \cE_3}) &= \chJH{E_{\zeta}}  \sum_{\vy \in [\vz + \chJH{B_H(L)}] \setminus \chJH{B_{H}(k)}} \left[\prob(\chJH{\cE_1 \cap \cE_2}) - \prob(\chJH{\cE_1 \cap \cE_2 \cap \cE_3^c}) \right]\nonumber\\  &\geq c L^2 \prob(\chJH{\cE_1}) - C L^2 K^{6-d} \log^8 K\, \prob(\chJH{\cE_1})\ ,
\end{align}
where we have used Claim \ref{clam:twoE} for the inequality. Taking $K$ sufficiently large and using the uniformity of the constants $c$, $C'$ establishes \eqref{eq:unifx}.
\end{proof}

We will now complete the proof of the first moment bound \eqref{eq:mombd1} from Theorem \ref{thm:mombd1} using Claim \ref{clam:threeE}. We claim that for any pair $(\vz,\vy)$ with $\vz \in S'(k)$ and $\vy \in [\vz + B_H(L)] \setminus B_{H}(k)$,
\begin{align}\nonumber
  \prob\big( (\vz,\vy) \text{ is } &(k,L,K)-\text{admissable and } X^{K-reg}(k) = M\big)\\
&\geq c(K) \prob\left(\cE_1 \cap \cE_2 \cap \cE_3\right)\label{eq:modE}
\end{align}
for a constant $c = c(K) > 0$, for all $K$ larger than some constant (depending only on the dimension $d$).  The bound of \eqref{eq:modE} is uniform in $n$, $k$, $L$, and $M$ as in the statement of Theorem \ref{thm:mombd1}, where $\tz$ is chosen for $\vz$ according to Claim \ref{clam:threeE} (note $\vz,\tz$ appear as arguments in the $\cE_i$ events on the right-hand side). The proof of \eqref{eq:modE} is via an edge modification argument similar to the one used to prove \eqref{eq:modmapped}, so we do not detail it here. Roughly speaking, one must open edges to connect $\vz$ to $\tz$ in a way that guarantees the pivotality of $\{\vz, \vz'\}$ without, for instance, changing the condition $\vz \in REG(K,k)$ guaranteed by $\cE_1$.

Given \eqref{eq:modE}, the conclusion of the proof is immediate. Summing the bound over $\vy \in [\vz + B_H(L)] \setminus B_{H}(k)$ and using Claim \ref{clam:threeE} gives
\begin{align*}
\sum_{\vy \in [\vz + B_H(L)] \setminus B_H(k)}\prob\big( (\vz,\vy) \text{ is } &(k,L,K)-\text{admissable and } X^{reg}(k) = M\big) \geq  c L^2 \prob(\cE_1)\ .
\end{align*}
Summing now over $\vz$ in the above gives a lower bound $c M L^2 \prob( X^{K-reg}(k) = M )$, since on $\cE_1$ we have $X^{K-reg}(k) = M$ definitionally.
\end{proof}

\begin{proof}[Proof of Lemma \ref{thm:mombd2}]
We abbreviate $Y = Y(k, L, K)$ and $\mathbf{1}_{M}$ for $\mathbf{1}_{X^{K-reg}(k) = M}$ and write
\begin{align}\label{eq:beftyppair}
  \E \left[Y^2 \mathbf{1}_{M}  \right] = \sum_{\substack{\vz_1,\vy_1 \\ \vz_2, \vy_2}} \prob\left((\vz_1,\vy_1)\text{ and } (\vz_2,\vy_2) \text{ are } (k,L,K)-\text{admissable} \right).
\end{align}
A typical term of the above sum can be written as
\begin{align}
\label{eq:typpair}
\sum_{\cC} \prob(C_{B_H(k)}(0) = \cC) \prob\left(\vy_i \sa{\Anns(n/2,4n)} \vz_i, \, \vz_i' \text{ pivotal for } \{0 \lra \vy_i\},\, i = 1, 2  \mid  C_{B_H(k)}(0) = \cC \right)  
\end{align}
where $\cC$ is such that conditions 1 -- 3 of the definition of admissability hold for the given $\vz_1$ and $\vz_2$ (note that these depend only on $C_{B_H(k)}(0)$). We consider first the case that $\vz_1 \neq \vz_2$ and $\vy_1 \neq \vy_2$.

On the event $\{\vy_i \sa{\Anns(n/2,4n)} \vz_i, \, \vz_i' \text{ pivotal for } \{0 \lra \vy_i\},\, i = 1, 2\}  \cap\{ C_{B_H(k)}(0) = \cC \}$ we claim there exist disjoint open paths $\gamma_1$ (resp. $\gamma_2$) connecting $\vy_1$ to $\vz_1'$ (resp. $\vy_2$ to $\vz_2'$) and avoiding $\cC$. To choose $\gamma_1$, consider any path $\sigma$ from $0$ to $\vy_1$. Since $\{\vz_1, \vz_1'\}$ is pivotal for the connection, this path passes through $\vz_1'$; the path must subsequently never intersect $\cC$ (otherwise $\{\vz_1, \vz_1'\}$ could be bypassed, contradicting pivotality), and so the terminal segment of $\sigma$ starting from $\vz_1'$ may be used as $\gamma_1$. If one chooses $\gamma_2$ similarly, we see that necessarily $\gamma_1 \cap \gamma_2 = \varnothing$. Indeed, if $\gamma_1$ and $\gamma_2$ intersected at some $\vv$, then following $\gamma_2$ from $\vy_2$ to $\vv$ and then following $\gamma_1$ from $\vv$ to $\vz_1'$ (or following $\gamma_1$ from $\vy_1$ to $\vv$ and then following $\gamma_2$), one sees that one of the edges $\{\vz_i, \vz_i'\}$ is not pivotal, a contradiction.

Having found such $\gamma_1$ and $\gamma_2$, one sees that when $\vz_1 \neq \vz_2$ and $\vy_1 \neq \vy_2$, the conditional probability in \eqref{eq:typpair} is at most
\[\prob(\vy_1 \lra \vz_1' \text{ off }\cC) \prob(\vy_2 \lra \vz_2' \text{ off }\cC) \leq \Cr{c:twopthigh}^2 \|\vz_1' - \vy_1\|^{2-d} \|\vz_1' - \vy_1\|^{2-d}\ . \]
Summing the above over $\vy_1 \neq \vy_2$ gives a uniform upper bound of $C L^4$. Putting this in \eqref{eq:typpair} and performing the sum over $\cC$, then doing an additional sum over $\vz_1 \neq \vz_2$ gives
\begin{equation}\label{eq:typpair2}\sum_{\substack{\vz_1 \neq \vz_2 \\ \vy_1 \neq \vy_2}} \prob\left((\vz_1,\vy_1)\text{ and } (\vz_2,\vy_2) \text{ are } (k,L,K)-\text{admissable} \right) \leq C M^2 L^4 \prob(X^{K-reg}(k) = M)\ .\end{equation}

When summing over terms in \eqref{eq:beftyppair} where $\vz_1 = \vz_2$, one is essentially computing an upper bound of the second moment of the cluster size of $\vz_1$; the resulting bound is $C M L^6 \prob(X^{K-reg}(k) = M)$. Since $M \geq L^2 /2$, this sum has an upper bound identical to that in \eqref{eq:typpair2}, completing the proof.

\end{proof}




Given \eqref{eq:mombd1} and \eqref{eq:mombd2}, Theorem \ref{thm:mainextend} now follows by a second moment argument similar to the one immediately following Lemma \ref{lem:yrestfirst} above. \qed

\section{Upper bound on $\pi_H(n)$}\label{sec:ubonearm}
We will prove the claim $\pi_H(n) \leq Cn^{-3}$ from part (a) of Theorem \ref{Critical Exponents} using two main ideas. The first main idea is an upper bound on the cardinality of $C_H(0) \cap Ann_H(n,2n)$ which gives some information about scaling in large clusters, and plays the role that knowledge of the cluster size exponent $\zeta$ would otherwise play (recall we have not yet proved part (c) of Theorem \ref{Critical Exponents}). A key ingredient is a mass-transport inequality, which controls the number of large half-space clusters.
The second main idea is an inductive argument which allows us to ``bootstrap'' control of $\pi_H(2n)$ from $\pi_H(n)$. This argument is based on a lemma which is similar in spirit to Lemma 2.3 of \cite{KN11}, with some major differences. These reflect the different geometry of $\Zd_+$ and the fact that we cannot use the two-point function or size exponents --- which were used in \cite{KN11} ---  having not yet proved parts (b) or (c) of Theorem \ref{Critical Exponents}.

 Recall the definition of a mass-transport rule from Section \ref{sec:masstrans}. In proving the upper bound for $\pi_H(n)$, we fix a particular $\mass$ once and for all for each fixed value of $n$:
\[\mass(0,\vx) = \begin{cases} 1 \quad &\text{if }\, \vx \in \mathsf{A}^{out}_0(n, n, 2n)\\
0 \quad &\text{otherwise.}\end{cases} \]
The set $\mathsf{A}_0^{out}$ was defined at \eqref{eq:aoutdef}.

The bound we will need for proving our main theorem comes from a  comparison of asymptotics for $\E \send$ and $\E \get$. Let $\kappa > 0$ be arbitrary (in practice, typically small). We define the event
\begin{equation}\label{eq:badsendevent}
A(\kappa):= \left\{\send \geq \kappa n^4\right\}\ . \end{equation}
By the definition \eqref{eq:badsendevent},
\begin{equation}
  \label{eq:badsendlb}
  \E \send \geq \kappa n^4 \prob(A(\kappa))\ .
\end{equation}
An upper bound on $\E\send$ follows via Theorem \ref{thm:mainextend};  Lemma \ref{lem:transport} and \eqref{eq:badsendlb} then show a corresponding upper bound for $\prob(A(\kappa))$. This is encapsulated in the following lemma.

\begin{lem}
\label{lem:get}
  There exists a $C$ such that, uniformly in $n$,
  \begin{align}
      \E \, \get \leq  C n\ .\label{eq:get}
  \end{align}
  In particular, we have the following bound uniformly in $\kappa$ and $n$:
  \begin{equation}
    \label{eq:badsendub}
    \prob(A(\kappa)) \leq \frac{C}{\kappa n^3}\ .
  \end{equation}
\end{lem}

\begin{proof}

Note that $0$ receives mass from $\vx$ if and only if both i) $0 \in \vx + Ann_H(n,3n)$ and ii) $\{0 \stackrel{\vx+\Anns(n/2,4n)}{\xleftrightarrow{\hspace*{1cm}}} \vx \}$ (recall $Ann_H(\ell_1, \ell_2) = B_H(\ell_2) \setminus B_H(\ell_1)$). The set of $\vx$ which satisfy the nonrandom condition i) is exactly $-Ann_H(n,3n)$. We break $\get$ into a sum of contributions over ``slices'' depending on $\ve_1$-distance, setting $T(j) = [-Ann_H(n,3n) \cap S(-j)]$ and
\begin{equation}\label{eq:justtofig}
Y(j)= \{\vx \in T(j), \ \vx \stackrel{\vx+\Anns(n/2,4n)}{\xleftrightarrow{\hspace*{1cm}}} 0\}\, , \quad n \leq j \leq 3n\ . \end{equation}
See Figure \ref{Fig: Dloc}(b) for a sketch.  In particular, $\get = \sum_j \#Y(j)$. We will use \eqref{eq:inextend} of Theorem \ref{thm:mainextend} to argue that if $Y(j)$ is too large, then $C(0) \cap [\vz + B_{-}(n/2)]$ is abnormally large, for some choice of $\vz \in T(j)$.

To that end, for $\vx \in T(j)$ we set
\begin{equation*}
  X_\vx(j) := \# \left\{\vy \in Y(j):\, \|\vy-\vx\| \leq n/4 \right\}\ .
\end{equation*}
There exists a deterministic set $\mathcal{S}_j \subseteq T(j)$ of no more than $5^{d-1}$ vertices such that,
for any $\vy \in T(j)$, there is an $\vx \in \mathcal{S}_j$ such that $\|\vy - \vx\| \leq n/4$. It follows that $\#Y(j) \leq \sum_{\vx \in \mathcal{S}_j} X_\vx(j)$. If we can show that
\begin{equation}\label{eq:XtoY}
\E X_\vz(j) \leq C\ \quad \text{uniformly in } n, \quad n \leq  j \leq 3n\ , \vz \in \mathcal{S}_j\end{equation}
we can immediately conclude that $\sum_j \E \#Y(j) \leq C n$ and the lemma is proved.


We now prove \eqref{eq:XtoY}.
  We will apply the third part of Theorem \ref{thm:mainextend}, but in shifted form. Define $D = \vx + \Anns(n/4,5n)$ and $Q = \partial_{-}[\vx + \Anns(n/4,5n)]$. Each $\vy \in Y(j)$ having $\|\vy - \vx\| \leq n/4$ also satisfies $\vy \in Q$.  The vertex $\vy$ is connected to $0$ by a path lying entirely in $\vy + \Anns(n/2,4n)$; in particular, this path lies in $D$. We therefore have the upper bound
\begin{equation}\label{eq:XtoXQ}
\prob(X_\vx(j) \geq M) \leq \prob(X_Q(D,0) \geq M)\ \text{ for all } M\ .
\end{equation}
We now bound the right-hand side of \eqref{eq:XtoXQ} by
\begin{align*}
\prob\left(  \#C(0) \cap [\vx + B(5 n)] \geq c_* n^2 M \right) + \prob\left( X_Q(D,0) \geq M, \,  \#C(0) \cap [\vx + B(5 n)] \leq c_* n^2 M \right),
\end{align*}
where $c_*$ is from Theorem \ref{thm:mainextend}. We note that the shifted analogue of $\# \mathsf{A}_0^{in}$ (shifted so $\vx$ plays the role of $0$)  is a lower bound for $\#C(0) \cap [\vx + B(5n)]$. Applying Theorem \ref{thm:mainextend} to the second term in the case when $M \geq n^2$ and rearranging, we see

\[ \prob(X_Q(D,0) \geq M) \leq c_*^{-1} \prob\left(  \#C(0) \cap [\vx + B(5 n)] \geq c_* n^2 M \right)\ \text{for all } M \geq n^2\ .\]

Integrating the above with respect to $M$, we find
\begin{align}
  \E X_Q(D,0) &\leq \E[X_Q(D,0); \,0 < X_Q(D,0) \leq n^2] + \E[X_Q(D,0); \,X_Q(D,0) > n^2]\nonumber\\
  &\leq n^2 \pi(n) + c^{-2}_* n^{-2} \E[\#C(0) \cap [\vx + B(5n)]; \, \#C(0) \cap [\vx + B(5n)]\geq c_* n^4]\ . \label{eq:lastterm}
\end{align}
Using \eqref{eq:onearmprob}, the first term of \eqref{eq:lastterm} is uniformly bounded by a constant. Using Lemma \ref{lem:momentaiz}, the second term of \eqref{eq:lastterm} is also bounded by a constant, giving $\E X_Q(D,0) \leq C$. In conjunction with \eqref{eq:XtoXQ}, this gives \eqref{eq:XtoY} and completes the proof of Lemma \ref{lem:get}.

\end{proof}

We continue with the proof of the upper bound from part (a) of Theorem \ref{Critical Exponents}, namely
\begin{equation}
  \label{eq:hstovary}
  \pi_H(n) \leq C n^{-3}\ .
\end{equation}
 The main remaining ingredient is the following lemma, which relates $\pi_H(2n)$ to $\pi_H(n)$.

\begin{lem}
  \label{lem:threeterms} There exist positive constants $C_1,\, C_2,\, c_1$ such that the following hold. For each $\lambda \in (0,1]$, there exists a constant $\varepsilon_0 = \varepsilon_0(\lambda) \in (0, 1)$ such that, for all $0 < \varepsilon \leq \varepsilon_0$,
  \begin{equation}
    \label{eq:mainiteqn}
    \pi_H(n(1+\lambda)) \leq \frac{C_1}{\varepsilon n^3} + C_2 \varepsilon^{3/5} \lambda^{-2} \pi_H(n) + (1-c_1) \pi_H(n)
  \end{equation}
uniformly in $n$ large relative to $\lambda$.
\end{lem}

We first prove \eqref{eq:hstovary} assuming the veracity of Lemma \ref{lem:threeterms} and then establish the lemma.

\begin{proof}[Proof of \eqref{eq:hstovary}]
  We begin by choosing $\lambda$ small enough to make the third term of \eqref{eq:mainiteqn} negligible. Namely, fix $0 <\lambda <1/2$ such that
  \begin{equation}
    \label{eq:makelambda}
    (1+\lambda) ^{3} (1-c_1) \leq (1-c_1/2)\ .
  \end{equation}
We will bootstrap a bound for $\pi_H(n)$ assuming it holds for $\pi_H(m),$ $m < n$. To this end, set $n_0:= \lceil 8\lambda^{-1} + c_*^{-1}\rceil$ and let $K > 0$ be a large constant such that
\begin{equation*}
  \pi_H(m) \leq K/m^3, \quad \text{for all } 1 \leq m \leq n_0\ .
\end{equation*}
We will also enlarge $K$ if necessary so that:
\begin{equation}
  \label{eq:makeKlarge}
  \max\left\{\frac{C_1 (36 C_2)^{5/3}}{c_1^{5/3}\lambda^{10/3}}\ ,\,
  \frac{C_1}{\varepsilon_0}\right\} \leq c_1 K / 64\ .
\end{equation}
 We show inductively that, for each $m \geq 0$, 
\begin{equation}
  \label{eq:hstovaryscale}
  \pi_H(n_0 (1+\lambda)^{m+1}) \leq K/(n_0(1+\lambda)^{m+1})^3 \quad \text{assuming } \pi_H(n_0(1+\lambda)^m) \leq K / (n_0(1+\lambda)^m)^3\ .
\end{equation}

Setting $n = n_0(1+\lambda)^m$, we apply \eqref{eq:mainiteqn} with the choice
\[\varepsilon = \min\left\{\varepsilon_0\ ,\, \frac{c_1^{5/3} \lambda^{10/3}}{\left(36 C_2 \right)^{5/3}} \right\}\ . \]
Note that $(1+\lambda)^3 \leq 8$ , so by the bound \eqref{eq:makeKlarge} we have
\[\text{First term of \eqref{eq:mainiteqn}} \leq \frac{c_1 K}{8[(1+\lambda)n]^3}\ . \]
A direct calculation similarly gives
\begin{align*}\text{Second term of \eqref{eq:mainiteqn}} &\leq \frac{2 c_1 K}{9[(1+\lambda)n]^3}\ ; \\
\text{Third term of \eqref{eq:mainiteqn}} &\leq \frac{K}{[(1+\lambda)n]^3} \left[1 - c_1 / 2 \right] \ .
\end{align*}
Pulling the three bounds above together completes the proof of \eqref{eq:hstovaryscale}.

To finish the argument for \eqref{eq:hstovary}, let $n > n_0$ be arbitrary and fix $m$ to be the largest integer such that $N := (1+\lambda)^m n_0 \leq n$. Note that, since $(1 + \lambda) \leq 2$, we have $N \geq n/2$. Using \eqref{eq:hstovaryscale}, \eqref{eq:makeKlarge}, and the monotonicity of $\pi_H$, we see
\begin{align*}
  \pi_H(n) \leq \pi_H(N) \leq K N^{-3} \leq 8 K n^{-3}\ ,
\end{align*}
establishling \eqref{eq:hstovary} with $C = 8K$.
\end{proof}

We now prove Lemma \ref{lem:threeterms}.
\begin{proof}[Proof of Lemma \ref{lem:threeterms}]
Fix $\lambda$ as in the statement of the lemma.  If $\varepsilon \leq n^{-1}$, then the above bound is simple using the one-arm exponent. Indeed, using \eqref{eq:onearmprob} we see
  \begin{align*}
    \pi_H((1+\lambda)n)\leq \pi_H(n) \leq  \pi(n) \leq \frac{\Cr{c:armhigh}}{n^2} \leq \frac{\Cr{c:armhigh}}{\varepsilon n^3}\ ,
  \end{align*}
and we are done. Otherwise, we will prove the bound by breaking up the connection event to $S'((1+\lambda)n)$, depending on the structure of the cluster of $0$.

  Recall the definition of the event $A(\kappa)$ in \eqref{eq:badsendevent} and the definition of the mass-transport rule $\mass$ above it.  We write $X(k)$ for  $X_Q(D,0)$ with $D = B_H(k)$ and $Q = S'(k)$, where $k$ is an integer satisfying $(1+\lambda/4) n \leq k \leq (1+\lambda/2)n$. The reason for emphasizing the $k$-dependence is that we will wish, in our definition of $\mathcal{D}_1$ below, to consider the first such integer value of $k$ for which $X(k)$ is small. The perspective here is similar to that of the proof of Lemma \ref{lem:get}, but from the perspective of vertices receiving mass instead of sending it. 

Given a value of $\varepsilon \in (n^{-1},1)$, we define $L = \varepsilon^{3/10} n$ and the events
\begin{align*}
  \cD_1(\varepsilon) &:= \left\{\exists k \in [n + \lambda n/4, n + \lambda n / 2]:\, X(k) \leq L^2 \text{ and } 0 \sa{\Zd_+} S'((1+\lambda)n) \right\}\ ;\\
  \cD_2(\varepsilon) &:= \left\{\forall k \in [n + \lambda n/4, n + \lambda n / 2],\, X(k) \geq L^2\ , \text{ and } \send < \varepsilon n^4 \right\}\ .
\end{align*}
The union bound and \eqref{eq:badsendub} give
\begin{align*}
  \pi_H((1+\lambda) n) &\leq \prob(A(\varepsilon)) + \prob(\cD_1(\varepsilon)) + \prob(\cD_2(\varepsilon))\\
&\leq \frac{C}{\varepsilon n^3} + \prob(\cD_1(\varepsilon)) + \prob(\cD_2(\varepsilon))\ .
\end{align*}
It suffices to show that the two $\prob(\mathcal{D}_i)$ terms above have upper bounds of the form of the second and third terms of \eqref{eq:mainiteqn}.

To bound the second term, let $I$ denote the (random) smallest integer value of $k$ as in the definition of $\cD_1$ such that $0 < X(k) \leq L^2$. Note that on $\cD_1$ we never have $X(k) = 0$, so we set $I = 0$ whenever some $X(k)$ is equal to zero. We explore the cluster of the origin in successive half-space boxes $B_H(k)$ until reaching $k = I$. At this point, the probability of further connection to $S'(n(1+\lambda))$ is, by Lemma \ref{lem:nofurther}:
\begin{align*}
  \prob(\cD_1(\varepsilon)) &= \sum_{k \in [n(1 + \lambda/4), n(1 + \lambda/2)]} \,\,\sum_{\cC \in \{I = k\}} \prob(0 \stackrel{\Z^d_+}{\longleftrightarrow} S'(n(1+\lambda)) ,\,  C_{B_H(k)}(0) = \cC)\\
  &\leq L^2 \pi(\lambda n / 4) \sum_{k \geq n(1 + \lambda/4)}\prob(I = k) \leq \frac{ C\varepsilon^{3/5}}{\lambda^2} \pi_H(n)\ ,
\end{align*}
where the second sum is over $\cC$ giving $I = k$ and where we have used \eqref{eq:onearmprob} to bound $\pi$.

The bound on $\cD_2$ is where we use Theorem \ref{thm:mainextend}, namely \eqref{eq:outextend}.
Since $\varepsilon > n^{-1}$, we have $L \geq n^{7/10}$, and so our choice of $L$ from above is a valid choice of $L$ in the statement of the theorem.

To set up our application of Theorem \ref{thm:mainextend}, we consider a sequence of values of $k$ and corresponding annular regions in which extensions can be made. For each integer $i \in [0, \frac{\lambda}{4} \varepsilon^{-3/10}]$ set $k_i = (1+\lambda/4)n + i L$ and note that $(1+\lambda/4) n \leq k_i \leq (1 + \lambda/2) n$. Recall that $c_*$ is the constant from Theorem \ref{thm:mainextend} and set 
\[\mathfrak{I} = \{i: \, X(k_i) > L^2,\, \#\mathsf{A}_0^{out}(n,k_i, L) < c_* L^4 \}\ . \]
Here we recall the notation
\[\mathsf{A}_0^{out}(n,k, L ) :=  C_{\Anns(n/2,4n)}(0) \cap Ann_H(k,k+L)\ . \]

Note that (by disjointness of the annuli $Ann_H(k_i, k_i + L)$)
\[c_{*} L^4 \# \{i: \, \#\mathsf{A}_0^{out}(n,k_i, L) \geq c_* L^4 \} \leq \send,\]
and so on $\cD_2(\varepsilon)$ we have
\begin{equation}
  \label{eq:shellindbd}
  \# \{i: \, \#\mathsf{A}_0^{out}(n,k_i, L) \geq c_* L^4 \} \leq \frac{1}{c_* \varepsilon^{1/5}}\ .
\end{equation}
In particular, on $\cD_2$, the cardinality of $\mathfrak{I}$ must be large; namely,
\begin{equation}
  \label{eq:Imustbelarge}
  \text{on }\cD_2(\varepsilon),\quad \#\mathfrak{I} \geq  \left\lfloor (\lambda/4) \varepsilon^{-3/10} \right\rfloor - c_*^{-1} \varepsilon^{-1/5}\ .
\end{equation}

On the other hand, using Theorem \ref{thm:mainextend} on each value of $i$ and summing, we have
\begin{equation}
  \label{eq:shellindbd2}
  \E \# \mathfrak{I} \leq (1-c_*) \pi_H(n) [(\lambda/4) \varepsilon^{-3/10} + 1] \ .
\end{equation}
We may now apply Markov's inequality with the bound \eqref{eq:shellindbd2} and compare to the lower bound for $\E \#\mathfrak{I}$ in terms of $\cD_2$ which follows from \eqref{eq:Imustbelarge}. This yields
\begin{equation}
\label{eq:Itoolarge}
\left[\, (\lambda/4) \varepsilon^{-3/10}  - c_*^{-1} \varepsilon^{-1/5} - 1 \right] \prob \left(D_2(\varepsilon) \right) \leq (1-c_*) \pi_H(n) [(\lambda/4) \varepsilon^{-3/10}+1]\ . 
\end{equation}
If $\varepsilon$ is sufficiently small (relative only to $\lambda$ and $c_*$), the left-hand side of \eqref{eq:Itoolarge} is at least
\[\frac{1-c_*}{1-c_*/2}\left[(\lambda/4) \varepsilon^{-3/10} + 1 \right] \prob(\cD_2(\varepsilon))\ .\]
Comparing the above to \eqref{eq:Itoolarge} gives $\prob(\cD_2(\varepsilon)) \leq (1-c_*/2) \pi_H(n)$ and completes the proof.
\end{proof}

\section{Half-space two-point function and cluster sizes}\label{sec:twoptclustend}
\subsection{Preliminaries and two-point function}
To better separate the proofs of the individual pieces, we restate the contents of part (b) of Theorem \ref{Critical Exponents}, consisting of bounds on the two-point function in $\Zd_+$.
\begin{thm}\label{thm:hstwopt}
	There exists a constant $C_1 > 0$ such that
	\begin{equation}
	\label{eq:twopthsupper}
	\tau_H(0,\vx) \leq C_1 \|\vx\|^{1-d}\qquad \text{ uniformly in } \vx \in \Zd_+\ .
	\end{equation}
	Fix $\varepsilon > 0$. Then there exists a constant $c_1 = c_1(\varepsilon)$ such that a matching lower bound holds for all points macroscopically far from $S(0)$, relative to $\varepsilon$:
	\begin{equation}
	\label{eq:twopthslower}
	\tau_H(0,\vx) \geq c_1 \|\vx\|^{1-d}\qquad \text{ uniformly in } \vx \in \Zd_+ \text{ with } x(1) \geq \varepsilon \|\vx\| .
	\end{equation}
	There exist constants $c_2, C_2 > 0$ such that the following holds uniformly in $\vy \in \Zd_+$ with $y(1) = 0$:
	\begin{equation}\label{eq:doublebdy}c_2 \|\vy\|^{-d} \leq  \prob(0 \sa{\Zd_+} \vy) \leq C_2 \|\vy\|^{-d}\ . \end{equation}
	
\end{thm}

Our proof of Theorem \ref{thm:hstwopt} relies crucially on the result of Theorem \ref{thm:boxcon} as input. We first prove a lemma which is in some respects a half-space analogue of Lemma \ref{OneCluster} and Lemma \ref{EXnbd}. For the statement, recall the definition $Rect(n) = [0, n] \times [-4n, 4n]^{d-1}$.

\begin{lem}\label{lem:fromloweradhoc}
	Let $D = Rect(n)$, and let $Q_1 = \partial_{\Zd_+} Rect(n)$ and $Q_2 = S(n) \cap Rect(n)$ (the ``top'' of  $Rect(n)$). Define $X_Q(D,0)$ as usual for $Q = Q_1, Q_2$.
	There exists $C_2 > 0$ such that, uniformly in $n$,
	\begin{equation}\E X_{Q_1}(D,0) \leq C_2 n^{-1}\ . 
	\label{eq:hsXexp}
	\end{equation}
	Recall the definition of $K_0$: the constant from Theorem \ref{thm:regthm}, chosen for the growing sequence $(Rect(n))_n$. There exist $\eta,\, c_2 > 0$ and such that the following holds uniformly in $n$ and in $K > K_0$:
	\begin{equation}
	\label{eq:hsXdens} \prob\left(\eta n^2 < X_{Q_2}^{K-reg}(D,0) \leq X_{Q_1}(D,0) <  \eta^{-1} n^2 \right) \geq c_2 n^{-3}\ . \end{equation}
\end{lem}

\begin{proof}
	We first show the bound on the expectation. By Lemma \ref{lem:get} (recall the notation of $\mathsf{A}^{out}$ defined before Theorem \ref{thm:mainextend}), we have
	\begin{equation}
	\label{eq:masstosf}
	\E \mathsf{A}^{out}_0(4n,4n,8n) \leq C n\ .
	\end{equation}
	By \eqref{eq:rectextend} from Theorem \ref{thm:mainextend}, we have
	\[\E[\mathsf{A}^{out}_0(4n,4n,8n) \mid X_{Q_1}(D,0)] \geq  c_*^2 X_{Q_1}(D,0) n^2  \quad \text{on } \{X_{Q_1}(D,0) \geq n^2/2\}\ . \]
	Combined with \eqref{eq:masstosf}, the above gives
	\[\E [X_{Q_1}(D,0); X_{Q_1}(D,0) \geq n^2/2] \leq C n^{-1}\ . \]
	On the other hand,
	\[\E[X_{Q_1}(D,0); X_{Q_1} < n^2/2] \leq (n^2/2) \pi_H(n) \leq C n^{-1}\ , \]
	where we have used part (a) of Theorem \ref{Critical Exponents}. This completes the proof of \eqref{eq:hsXexp}.
	
	To show \eqref{eq:hsXdens}, we note that by Theorem \ref{thm:regthm} it suffices to show
	\begin{equation}
	\label{eq:hsxsuff} \prob\left(\eta n^2 < X_{Q_2}(D,0) < X_{Q_1}(D,0)< \eta^{-1} n^2 \right) \geq c n^{-3} \text{ for all $n$,} \end{equation}
	for some $c, \eta > 0$. By \eqref{eq:armtofarside}, we have $\prob(X_{Q_2} > \eta n^2) \geq c n^{-3}$ for some fixed small $c$ (independent of $\eta$ and $n$) for $\eta$ sufficiently small. By \eqref{eq:hsXexp} and the Markov inequality, $\prob(X_{Q_1}(D,0) > \eta^{-1} n^2)$ is at most $C \eta n^{-3}$. 
	
	Bounding the probability in \eqref{eq:hsxsuff} by $\prob(X_{Q_2}(D,0) > \eta n^2) - \prob(X_{Q_1}(D,0) > \eta^{-1} n^2)$ and taking $\eta$ sufficiently small completes the proof.
	
	
\end{proof}

\begin{proof}[Proof of \eqref{eq:twopthsupper} and the upper bound of \eqref{eq:doublebdy}]
	We first prove \eqref{eq:twopthsupper}. Let $8n = \|\vx\|$. Note that if $0 \sa{\Zd_+} \vx$, there exists a $\vz \in \partial_{\Zd_+} Rect(n)$ such that $\{0 \sa{Rect(n)} \vz \} \circ \{\vz \lra \vx\}$ occurs. Taking a union bound and using the BK inequality gives
	\begin{align} \prob(0 \sa{\Zd_+} \vx) &\leq \sum_{\vz \in \partial_{\Zd_+}Rect(n)} \prob(0 \sa{Rect(n)} \vz) \prob(\vz \lra \vx) \label{eq:twiceslice}\\
	&\leq C n^{2-d} \sum_{\vz \in \partial_{\Zd_+}Rect(n)} \prob(0 \sa{Rect(n)} \vz) = C n^{2-d} \E X_{Q_1}(D,0) \label{eq:twiceslice2} \end{align}
	with $D = Rect(n)$ and $Q_1 = \partial_{\Zd_+}Rect(n)$, and where we have used \eqref{eq:twopt}. Applying \eqref{eq:hsXexp} completes the proof.
	
	The upper bound of \eqref{eq:doublebdy} follows from a decoupling argument similar to the one used for \eqref{eq:twopthsupper}, this time using \eqref{eq:twopthsupper} as input. As before, letting $8n = \|\vy\|$, for $0 \sa{\Zd_+} \vy$ to hold, there must be a $\vz \in \partial_{\Zd_+}Rect(n)$ such that $\{0 \sa{Rect(n)} \vz\} \circ \{\vz \sa{\Zd_+} \vy\}$ holds. This gives \eqref{eq:twiceslice} with $\vx$ replaced by $\vy$ and the connection from $\vz$ to $\vy$ restricted to $\Zd_+$. Now the same reasoning used to produce \eqref{eq:twiceslice2}, but now using the upper bound from \eqref{eq:twopthsupper} to estimate $\prob(\vz \sa{\Zd_+} \vy)$, gives the analogue of \eqref{eq:twiceslice2}, with $C n^{2-d}$ replaced by $C n^{1-d}$. Using \eqref{eq:hsXexp} as before completes the proof.
	
\end{proof}

\begin{proof}[Proof of \eqref{eq:twopthslower} and the lower bound of \eqref{eq:doublebdy}]
	We first prove \eqref{eq:twopthslower}. The argument is a modification of the proof of Theorem \ref{thm:boxcon}: roughly, we condition on $0$ having an arm to distance of order $n \approx \|\vx\|$, and then show an open connection from $\vx$ to this arm can be made.  There are three major modifications. First, if the arm from $0$ terminated too close to $S(0)$ (more carefully speaking: if $C_{B_H(n)}(0) \cap S'(n)$ had too few vertices at macroscopic distance from $S(0)$), this connection would not be possible; because of the lack of symmetry in the half-space, we must resort to the second part of Lemma \ref{lem:fromloweradhoc} to direct this arm. Second, there is no inductive improvement needed in the argument. 
	Third, we must rely on the result of Theorem \ref{thm:boxcon} as input to insure the further connection to $\vx$ does not cross the half-space boundary (the earlier argument required only information about the unrestricted $\tau$ as input in the base case).
	
	Fix $\varepsilon >0$ and suppose $\vx \in Ann_H(8n,16n)$ with $x(1) \geq \varepsilon n$. Let $D = Rect(n)$, and let $X_{Q_2} = X_{Q_2}(D,0)$ be as in the statement of Lemma \ref{lem:fromloweradhoc}. Let $K > K_0$ be fixed, to be chosen.   Define $Y_{Q_2}^K = Y_{Q_2}^K(\vx)$ to be the number of $\vz \in Q_2$ such that a) $\vz \sa{D} 0$, b) $\vz \in REG_D(K)$ (recall Definition \ref{RegDef}), and c) the edge $\{\vz, \vz + \ve_1\}$ is pivotal for $\{0 \sa{\Zd_+} \vx\}$. As in the proof of Lemma \ref{lem:yrestfirst}, we have $Y_{Q_2}^K \leq 1$ a.s., since no pair of vertices $\vz_1 \neq \vz_2$ can simultaneously satisfy parts a) and c) of the definition.
	
	Defining $B_\eta$ to be the event in \eqref{eq:hsXdens}, we argue that for $K > K_0$ fixed sufficiently large,
	\begin{equation}
	\label{eq:Ybdyre}
	\E[Y_{Q_2}^K; X_{Q_2}^{K-reg} = N, B_\eta ] \geq c n^{4-d} \prob(X_{Q_2}^{K-reg} = N; B_\eta)\quad \text{ for } c = c(K),
	\end{equation}
	uniformly in $\vx \in Ann_H(8n, 16n)$ with $x(1) \geq \varepsilon n$ and in $\eta n^2 \leq N \leq \eta^{-1} n^2$. Set $\widetilde D_\vz := \vz + [K/2, K]^d$ and let $\tz$ range over vertices of $\widetilde D_\vz$; define $R_n = B_H(0, 20n).$  We show \eqref{eq:Ybdyre} by defining events
	\begin{align*}
	\cE_1(\vz, N) &= B_\eta \cap \{0 \sa{D} \vz, \vz \in REG_{D}(K)\text{ and } X_{Q_2}^{K-reg} = N\}\ ;\\
	\cE_2(\vx, \tz, \vz) &= \{\tz \sa{R_n} \vx \text{ off } C_D(\vz) \}; \quad
	\cE_3(\vz, \tz) =	\{C(\vz) \cap C(\tz) = \varnothing\} \ .		
	\end{align*}
	
	Similar arguments to those of Claim \ref{clam:e2} show that we can choose $K > K_0$ and find a constant $c > 0$ such that the following holds: for each $n$, each $\vx \in Ann_H(8n, 16n)$ with $x(1) \geq \varepsilon n$, and each $\vz \in Q_2$, there is a $\tz \in \widetilde D_\vz$ such that
	\begin{equation} \label{eq:threeEHSt}
	\prob(\cE_1(\vz, N) \cap \cE_2(\vx, \tz, \vz) \cap \cE_3(\vz, \tz)) \geq c n^{2-d} \prob(\cE_1(\vz, N))\ .
	\end{equation}
	A main complication in proving \eqref{eq:threeEHSt}, compared with the proof of Claim \ref{clam:e2}, comes in the bound on $\prob(\cE_1 \setminus \cE_2)$. Namely: for the analogue of \eqref{E2bd}, we bound, on the event $C_{D}(0) = \cC$,
	\begin{align}
	\prob(\tz \sa{R_n} \vx \text{ off } \cC) \geq \prob(\tz \sa{R_n} \vx) - \sum_{\vy \in \cC} \prob(\{\tz \lra \vy \} \circ \{\vy \lra \vx\} )\ .\label{eq:sumtoredon}
	\end{align}
	To show the first term of the above is at least $c(\varepsilon) n^{2-d}$ using Theorem \ref{thm:boxcon}, we use crucially the fact that $\vz$ is macroscopically distant from $S(0)$. This necessitates the condition $\vz \in Q_2$, and this ultimately requires our arm-directedness statement \eqref{eq:armtofarside}.  The second term of \eqref{eq:sumtoredon} can be bounded similarly to before: the probability that $\vy \lra \vx$ is of order $n^{2-d}$, and the sum of probabilities $\prob(\tz \lra \vy)$ is small for $K$ large using the regularity in $\cE_1$.
	
	Having established \eqref{eq:threeEHSt}, we note that an edge-modification argument again gives the existence of a constant $c_1 = c_1(K)$ such that
	\begin{align*}\prob(\vz \text{ is counted in } Y_{Q_2}^K; B_{\eta}\cap\{ X_{Q_2}^{K-reg} = N\} ) &\geq  c_1 \prob(\cE_1(\vz,N) \cap \cE_2 (\vx, \tz, \vz) \cap \cE_3(\tz, \vz))\\
	&\geq c n^{2-d} \prob(\cE_1(\vz, N))\ ,
	\end{align*}
	where $\tz$ is chosen so that \eqref{eq:threeEHSt} holds. Summing over $\vz \in Q_1$, we get
	\begin{align*}\E[Y_{Q_2}^K;\, B_\eta \cap\{ X_{Q_2}^{K-reg} = N\}]
	&\geq c n^{2-d} \sum_{\vz \in Q_2}\prob(\cE_1(\vz, N))\\
	&\geq c n^{4-d} \prob(X_{Q_2}^{K-reg} =N, B_\eta)\ ,
	\end{align*}
	which is \eqref{eq:Ybdyre}.
	
	Having established \eqref{eq:Ybdyre}, we move to complete the proof of \eqref{eq:twopthslower}. Note that $\prob(0 \sa{\Zd_+} \vx) \geq \prob(Y_{Q_2}^{K}  > 0)$. We use a conditional second-moment argument to bound the latter probability. The fact that $Y_{Q_2}^{K} \leq 1$ a.s., and an argument similar to the one used to show (2A) of Lemma \ref{lem:yrestfirst}, give
	\begin{equation}
	\label{eq:ysquaredagain}
	\E[\left(Y_{Q_2}^{K}\right)^2 \mid X_{Q_2}^{K-reg} = N, B_\eta ] \leq C n^{4-d}\ .
	\end{equation}
	Combining \eqref{eq:ysquaredagain} with \eqref{eq:Ybdyre}, we find
	\begin{align*}\prob(Y_{Q_2}^K > 0 \mid X_{Q_2}^{K-reg} = N, B_\eta) &\geq \frac{\E[Y_{Q_2}^{K} \mid X_{Q_2}^{K-reg} = N, B_\eta ]^2}{\E[\left(Y_{Q_2}^{K}\right)^2 \mid X_{Q_2}^{K-reg} = N, B_\eta ]}\\
	&\geq c n^{4-d}\quad \text{ uniformly in $n$ and $\vx \in Ann_H(8n, 16n)$ with $x(1) \geq \varepsilon n$.} \end{align*}
	Recalling that $B_{\eta}$ was the event in \eqref{eq:hsXdens} and applying the probability bound there, we see
	\[\prob(0 \sa{\Zd_+} \vx) \geq \prob(Y_{Q_2}^K > 0) \geq c n^{4-d} \prob(B_\eta) \geq c n^{1-d} \geq c \|\vx\|^{1-d} , \]
	completing the proof of \eqref{eq:twopthslower}.
	
	We now outline the proof of the lower bound of \eqref{eq:doublebdy}; the proof is similar to the above, so we describe only the major differences. Suppose $\vy$ has $y(1) = 0$ and $\vy \in Ann_H(8n, 16n)$. As before, we set $D = Rect(n)$ and let $X_{Q_2}(D,0)$ be as in Lemma \ref{lem:fromloweradhoc}, and we define $Y_{Q_2}^K$ exactly as before (with references to $\vx$ replaced by $\vy$). 
	
	The events $\cE_i$ are defined as previously. except in $\cE_2$ we ask instead that $\tz \sa{\Zd_+} \vx$ off $C_D(\vz)$. Estimates involving the probability of this connection are made using \eqref{eq:twopthslower} instead of the bound on the box-restricted two-point function; upper bounds on the probability of appropriate portions of large-loop connections are made using the upper bound of \eqref{eq:twopthsupper}. For instance, the right-hand side of  \eqref{eq:threeEHSt} is replaced by $c n^{1-d} \prob(\cE_1(\vz,N))$. This reflects the fact that \eqref{eq:sumtoredon} is replaced by
	\begin{equation}\prob(\tz \sa{\Zd_+} \vy \text{ off } \cC) \geq \prob(\tz \sa{\Zd_+} \vy) - \sum_{\zeta \in \cC} \prob(\{\tz \lra \zeta \} \circ \{\zeta \sa{\Zd_+} \vy\} )\ . \label{eq:sumtoredon2}\end{equation}
	The first term of \eqref{eq:sumtoredon2} is uniformly at least $c n^{1-d}$ by \eqref{eq:twopthslower}. $\prob(\zeta \sa{\Zd_+} \vy)$ is at most $C n^{1-d}$ by the upper bound of \eqref{eq:twopthsupper}, and again the sum of probabilities $\prob(\tz \lra \zeta)$ is small for $K$ large.
	
	Making similar adaptations to the remaining estimates, we find that the conditional (on $B_\eta$) first and second moments of $Y_{Q_2}^K$ are both of order $n^{3-d}$. A conditional second-moment argument as before gives
	\[ \prob(0 \sa{\Zd_+} \vy) \geq \prob(Y_{Q_2}^K > 0) \geq c n^{3-d} \prob(B_\eta) \geq c \|\vy\|^{-d}\ .\]

\end{proof}


\subsection{Cluster sizes}
We now prove part (c) of Theorem \ref{Critical Exponents}. For clarity, we restate the claim here as Theorem \ref{thm:hssize}.
\begin{thm}\label{thm:hssize}
	There exist constants $c, \, C > 0$ such that
	\begin{equation}
	\label{eq:hssizeasymp}c t^{-3/4} \leq \prob(\# C_{H}(0) > t) \leq C t^{-3/4}\ . \end{equation}
\end{thm}
\begin{proof}
	We begin by proving the first inequality. We will compute cluster size moments conditional on $H_n := \{0 \sa{\Zd_+} S'(n)\}$. Abbreviate $Y_n = \#\left[C_{H}(0) \cap Ann_H(n,2n) \right]$. We can lower bound the (conditional) first moment of $Y_n$ by considering only those $\vx$ having $x(1) \geq n$:
	\begin{align}\label{eq:hscmoment}
	\E [Y_n \mid H_n] \geq c n^{3} \sum_{\vx \in Ann_H(n, 2n)} \tau_H(0, \vx)  \geq c n^3 \sum_{\substack{\vx \in Ann_H(n, 2n):\\ x(1) \geq n }} \tau_H(0,\vx) \geq \sum_{\substack{\vx \in Ann_H(n, 2n):\\ x(1) \geq n }} c n^{4-d} \geq c n^4\ ,
	\end{align}
	where we have used \eqref{eq:twopthslower} and the asymptotics of $\pi_H$.
	
	We can upper-bound $Y_n^2$ by $(\#[C_{H}(0) \cap B_H(2n)])^2$. Writing the latter quantity as a sum and using \eqref{eq:twopthsupper} gives
	\begin{align}
	\nonumber
	\E((\#[C_{H}(0) \cap B_H(2n)])^2\mid H_n) &\leq C n^3 \sum_{\vx, \vy \in B_H(2n)} \prob(0 \sa{\Zd_+} \vx,\, 0 \sa{\Zd_+} 
	\vy)\\
	&\leq C n^3 \sum_{\substack{\vx, \vy \in B_H(2n) \\ \vz \in \Zd_+}} \prob\left(\left\{0 \sa{\Zd_+} \vz\right\} \circ \left\{\vz \lra  \vx \right\} \circ \left\{\vz \lra \vy\right\}\right)\nonumber\\
	&\leq C n^3 \sum_{\substack{\vx, \vy \in B_H(2n) \\ \vz \in \Zd_+}} \|\vz\|^{1-d} \|\vz-\vx\|^{2-d} \|\vz-\vy\|^{2-d} \leq C n^{8}\ .\label{eq:hs2nde}
	\end{align} 
	Using the Paley-Zymund-inequality with \eqref{eq:hscmoment} \eqref{eq:hs2nde}, we find that there is a constant $c > 0$ such that, uniformly in $n$,
	\[\prob(Y_n > c n^4 \mid H_n) \geq c\ . \]
	Using the fact that $\prob(H_n) = \pi_H(n) \geq c n^{-3}$ gives $\prob(Y_n \geq c n^4) \geq c n^{-3}$. Since $\# C_{H}(0) \geq Y_n$, setting $n = C t^{1/4}$ for $C$ sufficiently large completes the proof of the first inequality of \eqref{eq:hssizeasymp}.
	
	To prove the second inequality of \eqref{eq:hssizeasymp}, first note that a calculation similar to that in \eqref{eq:hs2nde} shows $\E((\#[C_{H}(0) \cap B_H(2n)])^2) \leq C n^{5}\ . $
	Using this fact and Chebyshev's inequality, we see that for each $m > 0$,
	\begin{align}
	\nonumber
	\prob(\# C_{H}(0) > t) &\leq \pi_H(m) + \prob(\# [C_{H}(0) \cap B_H(m)] > t )\\
	\nonumber &\leq C m^{-3} + C m^5 / t^2\ .
	\end{align}
	 Setting $m = t^{1/4}$ completes the proof. 	
\end{proof}

\section*{Acknowledgements} The authors thank L.-P. Arguin for helpful discussions.

\bigskip
\addtocontents{toc}{\protect\setcounter{tocdepth}{1}}
\bibliographystyle{amsplain}
\bibliography{PercolationBib}

\addtocontents{toc}{\protect\setcounter{tocdepth}{0}}


\end{document}